%% file: main_arxiv.tex
\documentclass[11pt]{scrartcl}
\input{formatting}

\title{Catching Rats in $H$-minor-free Graphs\thanks{Emails of the authors: \href{mailto:m.gorsky@pm.me}{m.gorsky@pm.me},
\href{mailto:stamoulis@irif.fr}{stamoulis@irif.fr},
\href{mailto:sedthilk@thilikos.info}{sedthilk@thilikos.info},
\href{mailto:wiederrecht@kaist.ac.kr}{wiederrecht@kaist.ac.kr}}}
\predate{}
\date{}
\postdate{}

\preauthor{}
\DeclareRobustCommand{\authorthing}{
	\begin{center}
		Maximilian Gorsky\thanks{Supported by the Institute for Basic Science (IBS-R029-C1).}~~\! \\
		{\small Discrete Mathematics Group, Institute for Basic Science (IBS), Daejeon, South Korea} \\
        \medskip
        Giannos Stamoulis  \\
		{\small Université Paris Cité, CNRS, IRIF, F-75013, Paris, France} \\
        \medskip
		Dimitrios M.\ Thilikos\thanks{Supported by the French-German Collaboration ANR/DFG Project UTMA (ANR-20-CE92-0027), the ANR project GODASse ANR-24-CE48-4377,  and by the Franco-Norwegian project PHC AURORA 2024-2025 (Projet n°\! 51260WL).}   \\
		{\small LIRMM, Univ Montpellier, CNRS, Montpellier, France} \\ 
        \medskip
		Sebastian Wiederrecht \\
		{\small KAIST, Daejeon, South Korea}
\end{center}}
\author{\authorthing}
\postauthor{}

\begin{document}
\maketitle

\begin{abstract}
\noindent We show that every $H$-minor-free graph that also excludes a $(k \times k)$-grid as a minor has treewidth/branchwidth bounded from above by a function $f(t,k)$ that is linear in $k$ and polynomial in $t\coloneqq |V(H)|$.
Such a result was proven originally by [\textsl{Demaine \& Hajiaghayi, Combinatorica, 2008}], where $f$ was indeed linear in $k$.
However the dependency in $t$ in this result was non-explicit (and huge).
Later, [\textsl{Kawarabayashi \& Kobayashi, JCTB, 2020}] showed that this bound can be estimated to be $f(t,k)\in 2^{\Ocal(t\log t)}\cdot k$.
Wood recently asked whether $f$ can be pushed further to be polynomial, while maintaining the linearity on $k$.
We answer this in a particularly strong sense, by showing that the treewidth/branchwidth of $G$ is in $\Ocal(gk + t^{2304}),$ where $g$ is the Euler genus of $H$.
This directly yields $f(t,k)= \Ocal(t^2k + t^{2304})$.

Our methods build on techniques for branchwidth and on new bounds and insights for the Graph Minor Structure Theorem (GMST) due to [\textsl{Gorsky, Seweryn \& Wiederrecht, 2025, arXiv:2504.02532}].
In particular, we prove a variant of the GMST that ensures some helpful properties for the minor relation.
We further employ our methods to provide approximation algorithms for the treewidth/branchwidth of $H$-minor-free graphs.
In particular, for every $\varepsilon > 0$ and every $t$-vertex graph $H$ with Euler genus $g$, we give a $(g+\varepsilon)$-approximation algorithm for the branchwidth of $H$-minor-free graphs running in $2^{\poly(t)/\varepsilon}\cdot \poly(n)$-time.
Our algorithms explicitly return either an appropriate branch-decomposition or a grid-minor certifying a negative answer.
\end{abstract}
\let\sc\itshape
\thispagestyle{empty}

\newpage
\tableofcontents
\newpage

\setcounter{page}{1}

\section{Introduction}\label{sec_intro}
Treewidth is one of the most extensively studied parameters in structural graph theory.
Together with its closely related parameter branchwidth (introduced in \cite{RobertsonS1991Graph}), treewidth has served as a cornerstone of the Graph Minors Series \cite{RobertsonS2004Graph} (see \zcref{sec_prelims} for the definitions of branch- and treewidth).
Given a graph $G$, we use $\tw(G)$ (resp.\ $\bw(G)$) to denote the treewidth (resp.\ branchwidth) of $G$.
Both parameters are defined via tree-like decompositions that capture how closely a graph resembles a tree in its topological structure.
From a structural perspective, these two parameters are known to reflect the same underlying notion of ``tree-likeness.''
Indeed, for every non-acyclic graph $G$, it holds that $\bw(G)-1 \leq \tw(G) \leq \lfloor \nicefrac{3}{2}\cdot \bw(G)\rfloor-1$ (see~\cite{RobertsonS1991Graph}).
Variants and generalizations of these parameters appear widely across a broad range of applications in structural graph theory~\cite{HarveyW2017Parameters,AmarFG2009Degree,JohnsonRST01Direct,HlinenyW06Matro,JeongST2015Maximum}.   
Their true significance, however, becomes most apparent in algorithmic graph theory, where many problems can be solved efficiently, via dynamic programming, when restricted to graphs of bounded treewidth (or branchwidth)~\cite{Courcelle1990Monadic,BoriePT92auto,ArnborgLS91easy}.
Moreover, since computing both parameters is \NP-complete~\cite{SeymourT1994Call,ArnborgCP87Compl}, extensive research has focused on the design of \FPT-algorithms (i.e., algorithms deciding whether $\tw(G)\leq k$ (resp.\ $\bw(G)\leq k$) in $f(k)\cdot n^{\Ocal(1)}$ time for some computable function $f$)~\cite{BodlaenderK96Effic,BodlaenderT1997Constructive,KorhonenL23Impro}, approximation algorithms~\cite{BodlaenderGHK95Approx,FeigeHL08impro}, and \FPT-approximation algorithms~\cite{FominK24Fast,Korhonen21Asignle}.   
A particularly noteworthy island of tractability for branchwidth is the class of planar graphs.
Although the branchwidth of planar graphs is unbounded, it can be computed in polynomial time, thanks to the celebrated ``Ratcatcher algorithm’’ of Seymour and Thomas \cite{SeymourT1994Call}.
This, in turn, yields a constant-factor approximation for treewidth on planar graphs.

\subsection{The Grid Theorem and its Applications}
One of the most fundamental results related to treewidth is the celebrated \textsl{Grid Theorem}, which asserts the existence of a function $h\colon\Nbbb\to\Nbbb$ such that if a graph excludes a $(k \times k)$-grid\footnote{The \emph{$(k \times k)$-grid} is the Cartesian product of two paths on $k$ vertices.} as a minor\footnote{A \emph{minor} of a graph $G$ is constructed by taking a subgraph of $G$ and ``contracting'' edges within it. Here, by contraction, we mean identification of two vertices that share an edge, followed by the deletion of any loops or parallel edges that may appear.}, then its treewidth (or branchwidth) is bounded above by $h(k)$.

This theorem was first established in \cite{RobertsonS1986Grapha}, and the function $h$ has since been progressively improved in subsequent works \cite{RobertsonST1994Quickly,LeafS2015Treewidth,KawarabayashiK2020Linear}. Most notably, Chuzhoy and Tan proved that $h(k) \in \mathcal{O}(k^9 \cdot \polylog(k))$ \cite{ChekuriC2016Polynomial,ChuzhoyT2021Tighter}, thereby establishing a polynomial upper bound. Since the $(k \times k)$-grid has treewidth at least $k$, the theorem yields a min-max equivalence between the treewidth of a graph $G$ and the largest integer $k$ such that $G$ contains a $(k \times k)$-grid minor. This maximum $k$ is called the \emph{biggest grid number} of $G$, denoted by $\bg(G)$.

The above min-max equivalence is especially useful in algorithm design. It enables the following win/win algorithmic strategy:  if a $(k \times k)$-grid minor is present, it can directly lead to a solution, while its absence guarantees bounded treewidth, enabling dynamic programming techniques. Consequently, improving bounds on the function $h$ in the Grid Theorem has direct and far-reaching algorithmic implications, as it refines parameter dependencies in algorithms that rely on such strategies.

It is known that, in general, the function $h$ cannot be improved beyond $\Ocal(k^2 \log k)$ (see e.g.\ \cite{Thilikos12Graphm}), although tighter bounds exist for special classes of graphs. In particular, when the input graph belongs to a minor-closed class, the relationship between treewidth (or branchwidth) and the biggest grid number becomes \textsl{linear}. This was first established by Demaine and Hajiaghayi \cite{DemaineH2008Linearity}, who proved the existence of a function $f\colon\Nbbb\to\Nbbb$ such that for every graph excluding a $t$-vertex graph $H$ as a minor, the following holds:
\begin{eqnarray}
\tw(G) ≤ f(t) \cdot \bg(G). \label{such_no}
\end{eqnarray}

This linear relationship has important algorithmic consequences. In particular, it enables sublinear treewidth bounds for many parameterized problems on $H$-minor-free graphs. To illustrate this, consider two representative examples: \textsc{Feedback Vertex Set} and \textsc{Longest Path}. In both problems, the input is a pair $(G, k)$. \textsc{Feedback Vertex Set} asks for a set of at most $k$ vertices that intersects all cycles in $G$, while \textsc{Longest Path} asks for a path with at least $k$ vertices in $G$.

Thanks to \eqref{such_no}, instances $(G, k)$ where $G$ excludes $K_t$ as a minor can be efficiently reduced\footnote{To see this, observe that a $(\lceil \sqrt{k} \rceil \times \lceil \sqrt{k} \rceil)$-grid minor in $G$ certifies that $(G, k)$ is a \textsf{no}-instance (resp.\ \textsf{yes}-instance) of \textsc{Feedback Vertex Set} (resp.\ \textsc{Longest Path}). Hence, one may assume that the input graph has treewidth sublinear in $k$.} to graphs with treewidth $\tw(G) = \Ocal_{t}(\sqrt{k})$.\footnote{By $g(n)=\Ocal_{k}(f(n))$ we mean that there is a function $h\colon\Nbbb\to\Nbbb$ such that $g(n)=\Ocal(h(k)\cdot f(n))$.}

Problems exhibiting this behaviour are typically called \textsl{bidimensional} (see \cite{DemaineFHT2005Subexponential}). The sublinear treewidth condition, combined with standard dynamic programming—often running in $2^{\Ocal(\tw)} \cdot n^{\Ocal(1)}$ or $2^{\Ocal(\tw \log \tw)} \cdot n^{\Ocal(1)}$ time—yields subexponential parameterized algorithms. For example, using techniques from \cite{CyganNPPRW22Solving} (see also \cite{BodlaenderCKN15Deterministic,DornPBF10Efficient,DornFT08Catalan,BasteST21Hitting}), both \textsc{Feedback Vertex Set} and \textsc{Longest Path} can be solved in $2^{\Ocal(\tw)} \cdot n^{\Ocal(1)}$ time. This, in turn, implies subexponential parameterized algorithms running in $2^{\Ocal_t(\sqrt{k})} \cdot n^{\Ocal(1)}$ time for both problems on $K_t$-minor-free graphs.

The need to establish relations like \eqref{such_no} between $\tw$ and $\bg$ extends beyond the design of subexponential parameterized algorithms (as discussed in \cite{DemaineFHT2005Subexponential}). Such relations have proven crucial in deriving linear kernels \cite{FominLST20Bidimensionality} and EPTASs \cite{FominLS18Excluded} for bidimensional problems (see also \cite{FominLST20Bidimensionality,Demaine10algo,DemaineHT06theb,DemaineFHT04Bidimensional,BasteT22Contraction}).

More broadly, these techniques have inspired a family of algorithms with running times of the form:
\begin{eqnarray}
2^{\Ocal_t(\sqrt{k})} \cdot \Ocal(n^c) \quad \text{or} \quad 2^{\Ocal_t(\sqrt{k} \log k)} \cdot \Ocal(n^c),\label{do_ops}
\end{eqnarray}
for problems on $H$-minor-free graphs.
In \eqref{do_ops} $c$ is typically some small integer constant. This phenomenon—termed the “square root phenomenon” by Marx~\cite{Marx13Thesquare}—is pervasive in the context of parameterized algorithms on minor-closed graph classes. Intuitively, such dependencies stem from the structural properties of $H$-minor-free graphs and have also been explored in broader graph classes \cite{BasteT22Contraction,BertheB0R24Subexponential,BertheBGR24Feedback}.

We assume henceforth that $t\coloneqq |V(H)|$ and that $n$ is the size of the input graph. 
Notice that in running times resembling those in \eqref{do_ops}, the influence of the excluded graph $H$ is ``hidden’’ in the $\Ocal_t$ notation. This has motivated efforts to determine \textsl{explicit} bounds for the function $f$ in \eqref{such_no}, as such bounds were not provided in the original result by Demaine and Hajiaghayi~\cite{DemaineFHT2005Subexponential}. In this direction, Kawarabayashi and Kobayashi~\cite{KawarabayashiK2020Linear} established that  $f(t) = 2^{\Ocal(t \log t)}$.

Recently, David Wood \cite{Wood2025Personal} posed the following question:
\begin{eqnarray}
\mbox{Is there a \textsl{polynomial} function $f$ such that for every graph $G$, we have $\tw(G) ≤ f(t) \cdot \bg(G)$?}\label{main_que}
\end{eqnarray}

\subsection{Combinatorial results}
In this paper, we resolve the question in \eqref{main_que} positively in a way that reveals the pivotal role of the Euler genus $g_H := \eg(H)$ of the excluded graph $H$. We prove that for every $H$-minor-free graph $G$,
\begin{equation}
\tw(G) = \Ocal(g_H \cdot \bg(G) + t^{2304}). \label{ral_iop}
\end{equation}

Our bound in \eqref{ral_iop} reveals that the contribution of the excluded graph $H$ in the aforementioned subexponential algorithms depends primarily on its Euler genus $g_H$, rather than just its size $t$. Specifically, our result implies that the aforementioned subexponential algorithms  run in time
\[
2^{\Ocal(g_H \cdot \sqrt{k})} \cdot \Ocal_t(n^c) \quad \text{or} \quad 2^{\Ocal(g_H \cdot \sqrt{k} \log k)} \cdot \Ocal_t(n^{c}),
\]
highlighting that the parameter dependence is fundamentally governed by the topological complexity of $H$, as captured by its Euler genus.

\paragraph{Excluding graphs and grids.}
Given that $\eg(H) = \Ocal(t^2)$, \eqref{ral_iop} can be simplified to the following.

\begin{theorem}\label{thm_amaingrid}
For every $t$-vertex graph $H$,  every $H$-minor-free graph $G$ that excludes 
a $(k\times k)$-grid as a minor, has treewidth/branchwidth  in $\Ocal(t^2k + t^{2304})$.
\end{theorem}

Also our proof of \eqref{ral_iop} yields the following the following stronger statement regarding induced minors, which are found by taking \textsl{induced} subgraphs and performing edge contractions.

\begin{theorem}\label{thm_inducedgrid}
If a graph $G$ excludes both $K_{t}$ and the  $(k\times k)$-grid as an induced minor, then 
it has treewidth/branchwidth  in $\Ocal(t^2k + t^{2304})$.
\end{theorem}

\paragraph{Apex graphs.}
An \emph{apex} graph is a graph in which there exists a vertex whose removal makes the graph planar.
Recently Hendrey and Wood \cite{HendreyW2025Polynomial} improved a well-known theorem on graphs of bounded radius $r$ excluding an apex graph with $t$ vertices as a minor, originally proven by Eppstein with worse bounds \cite{Eppstein2000Diameter}, by showing that such graphs have treewidth in $\Ocal^*(r^9t^{18})$, where the $\Ocal^*$ notation hides poly-logarithmic factors.
According to the results of Kawarabayashi and Kobayashi in \cite{KawarabayashiK2020Linear}, the dependency on $r$ can be pushed to be linear, though the bound features a factor that is \textsl{exponential} in $t$.
Using \zcref{thm_amaingrid} and the proofs presented in \cite{HendreyW2025Polynomial}, this can be improved in a straightforward fashion to the following.

\begin{theorem}
    For any apex graph $A$ with $t$ vertices, every $A$-minor-free graph $G$ of radius at most $r$ has treewidth in $\Ocal(t^4r + t^{2304})$.
\end{theorem}

\subsection{Approximation algorithms}
There is a long line of research on approximating treewidth and branchwidth. Since the two parameters are linearly related, any approximation algorithm for one immediately implies an approximation algorithm for the other.

An early result by Bodlaender, Gilbert, Hafsteinsson, and Kloks~\cite{BodlaenderGHK95Approx} established a polynomial-time $\Ocal(\log n)$-approximation algorithm for both parameters. Later, Feige, Hajiaghayi, and Lee~\cite{FeigeHL08impro} gave a randomized polynomial-time algorithm achieving an improved approximation factor of $\Ocal(\sqrt{\log n})$. On the negative side, computing treewidth within an additive error of $n^\varepsilon$ for some $\varepsilon > 0$ is \NP-hard~\cite{BodlaenderGHK95Approx}, and Austrin, Pitassi, and Wu~\cite{WuAPL15Inapproximability} showed that no constant-factor approximation exists under the Small Set Expansion Conjecture. More recently, Bonnet proved that approximating treewidth within a factor of $1.00005$ is \NP-hard~\cite{Bonnet25Treewidth}.

For restricted graph classes, better approximation algorithms are known. In particular, both parameters admit constant-factor approximations on asteroid-triple-free graphs~\cite{BouchitteT03Approximating}. Similarly, for $H$-minor-free graphs, Feige et al.~\cite{FeigeHL08impro} provided a polynomial-time $\Ocal(t^2)$-approximation algorithm when $H$ is a $t$-vertex graph. This raises the natural question of whether this dependence on $t$ can be improved.

The proof of \eqref{ral_iop}, that is the base of all our combinatorial results, is algorithmic. In \zcref{sec_Hminorfree} we derive from it the following approximation algorithm for treewidth/branchwidth.

\begin{theorem}
\label{main_alg}
There exists an algorithm that, given $k \in \Nbbb$, a $t$-vertex graph $H$ of Euler genus $g_H$, and an $H$-minor-free graph $G$, outputs either a $(k \times k)$-grid minor of $G$ or a tree/branch-decomposition of $G$ of width 
$\Ocal(g_H \cdot k + (g_H + 1) \cdot \poly(t))$, in $2^{\poly(t)} \cdot \poly(n)$ time.
\end{theorem}

By leveraging the randomized algorithm of Chekuri and Chuzoy~\cite{ChekuriC2016Polynomial}, the output guaranteed by \zcref{main_alg} can also be achieved with high probability in $(t+n)^{\Ocal(1)}$ time.

As a consequence, we obtain a $\Ocal(g_{H})$-approximation algorithm for computing $\bw$, $\tw$, and $\bg$ on $H$-minor-free graphs. Since the size of $H$ and its Euler genus $g_H$ can differ substantially (e.g., when $H$ is planar), this result constitutes a significant improvement over the previous $\Ocal(|V(H)|^2)$-approximation bounds for branchwidth and treewidth~\cite{FeigeHL08impro}.

Finally, we observe that—particularly for the branchwidth parameter—the approximation factor can be made \textsl{arbitrarily close} to the Euler genus $g_H$ of the excluded graph $H$.

\begin{theorem}
There exists an algorithm that, given a $t$-vertex graph $H$ with Euler genus $g_H$ and an $H$-minor-free graph $G$, outputs a value $b$ such that $b \leq \bw(G) \leq (g_H + \varepsilon) b$, in $2^{\poly(t)/\varepsilon} \cdot \poly(n)$ time.
\end{theorem}

Again, using the randomized algorithm of Chekuri and Chuzoy~\cite{ChekuriC2016Polynomial}, the above approximation can also be turned into a randomized one running in time $2^{\mathbf{poly}(\nicefrac{1}{\varepsilon})}\mathbf{poly}(t+n)$.

\paragraph{Future directions.}
In this paper we prove that $\tw$, $\bw$, and $\bg$ can all\footnote{Certainly, the same approximation applies for every graph parameter that is linearly equivalent to treewidth (see e.g., \cite{DvorakN19Treewidth,LardasPTZ23OnStrict,MescoffPT23Themixed,HarveyW2017Parameters}).} be approximated in polynomial time on $H$-minor free graphs with an approximation factor that depends on the Euler genus of the excluded graph $H$ rather than the size of $H$.
The natural question arising in the above context is the following: 

\begin{quote}
\textsl{Is the parameter of Euler genus of $H$ the optimal measure for the approximation factor?}
\end{quote}

We suspect that the answer to this question is negative.
However, we certainly believe that Euler genus is the ``most natural'' measure that we may imagine for this.

\paragraph{Related work.}
The linear dependence between $\tw$ and $\bg$, as denoted by~\eqref{such_no}, was revealed for the first time for minor-closed graph classes in \cite{DemaineH2008Linearity}, while for graphs of bounded genus this was already known in \cite{DemaineFHT2005Subexponential}.
These results germinated the first ideas on the approximability of treewidth/branchwidth: they can be traced back to the proofs of \cite{DemaineFHT2005Subexponential} for the case of bounded genus graphs, in \cite{DemaineHNRT04Approximation} for single-crossing minor-free graphs (using the structural theorem of \cite{RobertsonS91Excluding}), and to the proofs of \cite{DemaineH2008Linearity} and \cite{DornFT08Catalan} for $H$-minor free graphs (see also \cite{FeigeHL08impro}).
Refined structural theorems were also recently used in \cite{ThilikosW2025Approximating} for the derivation of an EPTAS for branchwidth for graph classes excluding a toroidal and a projective graph.
In our proofs, we use some of the tools presented in~\cite{ThilikosW2025Approximating}.

\subsection{Sketch of our proof}
For the proof of \zcref{thm_amaingrid}, which implies \eqref{ral_iop}, we work with branchwidth instead of treewidth, as branchwidth behaves better with embedded graphs.
We make strong use of the structure of $H$-minor-free graphs as revealed in \cite{RobertsonS2003Grapha} by the Graph Minor Structure Theorem (GMST).
This theorem indicates that $H$-minor-free graphs can be tree-decomposed into a collection of graphs, called \textsl{torsos}, that can be ``almost embedded'' in a surface where $H$ does not embed.
By ``almost embedded'' we mean that a considerable part of each torso $G_{t}$ is embedded ``up to 3-separations'' in the surface while the rest of it is either a small number of ``apices'' or can be drawn inside a bounded number of faces as ``vortices of small width''.
(We postpone the precise definitions to \zcref{sec_gmst}.)

For our purposes, we prove a special variant of the GMST that permits both the values of $\bw$ and $\bg$ on $G$ to be computed using the corresponding values of its torsos.
As already mentioned, each torso $G_{t}$ is partitioned into two parts: the part that is embedded ``up to 3-separations'', say $G_{t}'$, and the rest that corresponds either to the apices or the parts that are drawn inside the vortices.
Since the tools we want to apply later need an actual embedding, we use the quasi-4-connected decomposition of Grohe in \cite{Grohe2016Quasi4Connected} to get rid of the vague parts of $G_t'$ that develops at the site of the 3-separations.
The proof of our version of GMST uses as a departure point the recent version of GMST proven in \cite{GorskySW2025Polynomial} that ensures that both the number of apices or vortices, as well as the width of the vortices are all bounded by a polynomial function of the size of $H$.
Our proof is algorithmic, running in time $2^{\poly(t)}\cdot \poly(n)$, based on the running time of the GMST in \cite{GorskySW2025Polynomial}.

We next derive a constant-factor approximation algorithm for the branchwidth of each torso $G_{t}$ that either outputs a $(k\times k)$-grid minor of $G_{t}$, certifying that $\bw(G_{t}) \geq k$, or builds a branch-decomposition of width at most $\Ocal(gk + \poly(t))$, where $g$ is the Euler genus of the surface where $G_{t}'$ is embedded that, in turn, is bounded by the Euler genus of $H$. 
For this, we work with the graph $G_{t}'$ by removing from it additional apex vertices obtained by a surface-cutting approach along shortest non-contractible curves, using the algorithms of \cite{CabelloCL2012Algorithms}.
This procedure reduces $G_{t}'$ to a simpler graph $G_{t}''$ that is either embedded into a surface different than the sphere in a way that certifies the existence of a $(k\times k)$-grid using the results of \cite{DemaineFHT2005Subexponential} or to a graph embedded into a sphere.
In the latter case, we apply the Ratcatcher algorithm from \cite{SeymourT1994Call} and find {either a certificate of a $(k\times k)$-grid} or a special branch decomposition, called \textsl{sphere-cut decomposition} (see~\cite{RueST14Dynamic,SauT10Subexponential,DornFT08Catalan,DornPBF10Efficient}).
Using this decomposition we return the apices and vortices that are missing from the spherical embedding of $G_t''$ to $G_t$ and we do this with an \textsl{additive} cost on the width of the obtained branch-decomposition of $G_{t}$, where the part contributed by the apices and vortices can be guaranteed to be comparatively small.
The fact that this is true is a core result of \cite{ThilikosW2025Approximating}, though their proof is not constructive.
To be able to construct a branch-decomposition explicitly, we thus develop a constructive proof of this result along the way (see \zcref{thm_spherewithoutvorticeshighbw}).
This leads to the construction of a branch-decomposition of each torso $G_{t}$ of width in $\Ocal(t^2 k + \poly(t))$ in $2^{\poly(t)}\cdot \poly(n)$ time.

As the version of the GMST that we prove roughly permits the computation of both $\bw(G)$ and $\bg(G)$ from the corresponding values of its torsos, we derive that the same linear relation between branchwidth and biggest grid number holds also for $G$.
Once more we note that by using randomized algorithms due to Chekuri and Chuzoy~\cite{ChekuriC2016Polynomial}, we can push our results to be truly polynomial, at the cost of only receiving our desired results with high probability.
Notably, this change only impacts the proof of our variant of the GMST and leaves the rest of our methods untouched.

\section{Preliminaries}\label{sec_prelims}
We generally adhere to the notation for graphs defined in \cite{Diestel2010Graph}.
All graphs we consider are simple, meaning that they contain neither loops nor parallel edges, and contain at least one vertex.
We call a maximal induced subgraph of a graph without cut vertices a \emph{block}.
\paragraph{Surfaces.}
A \emph{surface} $\Sigma$ is a compact 2-dimensional manifold with or without boundaries.
The surfaces we consider are generally assumed to be connected.
Given a pair $(\mathsf{h},\mathsf{c}) \in \mathbb{N} \times [0,2]$, we let $\Sigma^{\mathsf{(\mathsf{h},\mathsf{c})}}$ be the surface without boundary created from the sphere by adding $\mathsf{h}$ handles and $\mathsf{c}$ crosscaps.
Dyck's theorem \cite{Dyck1888Beitraege,FrancisW1999Conways} tells us two crosscaps are equivalent to a handle in the presence of a third crosscap and thus our notation is sufficient to capture all surfaces without boundary.

Given a surface $\Sigma$, we add an \emph{open}, respectively \emph{closed hole} to $\Sigma$ by removing a closed, respectively open disk from $\Sigma$.
If $\Sigma$ is a surface (with holes), we let $\overline{\Sigma}$ be the surface resulting from gluing a closed, respectively open disk onto each open, respectively closed hole of $\Sigma$.
We let the \emph{genus} of $\Sigma$ be $2\mathsf{h} + \mathsf{c}$, where $\Sigma^{\mathsf{(\mathsf{h},\mathsf{c})}}$ is the surface to which $\overline{\Sigma}$ is isomorphic.
Given a surface $\Sigma$, we denote by $\mathsf{bd}(\Sigma)$ the \emph{boundary} of said surface.

A curve $C$ in a surface $\Sigma$ is called \emph{contractible} if it is simple, closed, and is continuously deformable to a point.
We call a curve \emph{non-contractible} if it is simple, closed, and not contractible.
Note that, if $\Sigma$ is not a sphere (with or without holes) and we let $\Sigma_1,\Sigma_2$ be the two components of $\Sigma - C$, for a non-contractible curve $C$,
then the minimum of the genus of $\Sigma_1$ and $\Sigma_2$ is less than the genus of $\Sigma$.

\paragraph{Embeddings.}
Given a graph $G$, we say that $G$ has an \emph{embedding} $\psi$ into a surface $\Sigma$, if $\psi$ is a function with domain $V(G) \cup E(G)$ such that
\begin{enumerate}
    \item $\psi(v)$ is a point of $\Sigma$ for each $v \in V(G)$,
    \item $\psi(u) \neq \psi(v)$ for all distinct $u,v \in V(G)$,
    \item $\psi(uv)$ is a simple curve in $\Sigma$ with the ends $\psi(u),\psi(v)$ for all $uv \in E(G)$,
    \item $\psi(uv) \cap \psi(wx) \subseteq \{ \psi(u), \psi(v) \}$, for all $uv, wx \in E(G)$, and
    \item $\psi(uv) \cap \psi(w) = \emptyset$ for all $uv \in E(G)$ and $w \in V(G) \setminus \{ u,v \}$.
\end{enumerate}
We let $\psi(G)$ be the union of all points and curves in its image and call the connected components of $\Sigma \setminus \psi(G)$ the \emph{faces} of $\psi$.
If every face of $\psi$ is homeomorphic to an open disk, we say that $\psi$ is a \emph{2-cell-embedding}.
We will generally work with graphs with a fixed embedding and thus identify the elements of $G$ with the points and curves in $\Sigma$ which $\psi$ maps to each other.
The minimum number of vertices of $G$ that is intersected by any non-contractible curve in $\Sigma$ is called the \emph{representativity} (or \emph{face-width}) of $\psi$.
We consider the representativity of an embedding in the sphere to be infinite.

Of course, given some embedding it may not be immediately clear what its representativity is.
For this purpose, Cabello, Colin de Verdi\`ere, and Lazarus provide us with a linear time algorithm that checks whether a given 2-cell-embedding has representativity at most $k$ \cite{CabelloCL2012Algorithms}.
They also prove the following result, which we will make frequent use of.

\begin{proposition}[Cabello, Colin de Verdière, and Lazarus \cite{CabelloCL2012Algorithms}]\label{prop_findnoncontrcurve}
    There exists an algorithm that, given a graph $G$ with a 2-cell-embedding of representativity $k$ in a surface $\Sigma$ of genus $g$ with $h$ holes, finds a non-contractible curve in $\Sigma$ intersecting the minimum number of vertices of $G$ in $\Ocal((g+h)\cdot k\cdot (n+m))$-time.
\end{proposition}

The \emph{Euler genus} $\mathsf{eg}(G)$ of $G$ is the minimum genus of a surface $\Sigma$ such that $G$ has an embedding into $\Sigma$.
Determining the Euler genus of a given graph is an NP-complete problem \cite{Thomassen1989Graph,Thomassen1993Triangulating}.
However, it is possible to efficiently determine whether a graph can be embedded in a surface of some fixed genus \cite{Mohar1999Linear,KawarabayashiMR2008Simpler}.
Notably, the embedding produced by the algorithm in \cite{Mohar1999Linear} is a 2-cell embedding as long as such an embedding can exist in the surface we ask the graph to be embedded in.
As a consequence we are able to use the following.

\begin{proposition}[Kawarabayashi, Mohar and Reed \cite{KawarabayashiMR2008Simpler}]\label{prop_findembedding}
    There exists an algorithm that, given a graph $G$ of Euler genus $g$, returns a 2-cell embedding of $G$ in a surface of genus at most $g$ in $2^{\mathsf{poly}(g)} n$-time.
\end{proposition}

The runtime of this algorithm is not polynomial in all of its parameters.
As a consequence, we sometimes split out our proofs to demonstrate that once a 2-cell embedding is given, our computations run in polynomial time.
See \cite{MoharT2001Graphs} for more details on surfaces and embeddings.

\paragraph{Cuts.}
Let $G$ be a graph and let $S \subseteq E(G)$.
If $S \subseteq E(G)$, the set $\partial_G(S)$ is the \emph{cut around $S$ (in $G$)}, which contains all vertices in $V(G)$ incident to edges in both $S$ and $E(G) \setminus S$.
We drop the index $G$ from this set if it is clear from the context which graph is meant.

\paragraph{Branchwidth.}
We let $L(T)$ be the set of leaves of the tree $T$.
The non-leaf vertices of a tree are called \emph{internal vertices}.
A tree is called \emph{ternary} if all of its internal vertices have degree 3.
Accordingly, a ternary tree with two leaves corresponds to a single edge and a ternary tree with a single leaf is simply an isolated vertex.

Given a ternary tree $T$ with a bijection $\tau$ from $E$ to $L(T)$, note that each edge $e \in E(T)$ induces a partitioning of $E$ into two non-empty sets $X_e,Y_e$ via $\tau$, where we note that $\partial(X_e) = \partial(Y_e)$.
If $E = E(G)$, then the \emph{branchwidth of $(T,\tau)$} is the maximum over all edges $e$ in $E(T)$ of $|\partial(X_e)|$ and the \emph{branchwidth $\bw(G)$ of $G$} is the minimum of the branchwidth over all ternary trees $T$ and bijections $\tau$ as above.
We call the pair $(T,\tau)$ a \emph{branch-decomposition} of $G$.

The structure of graphs with very small branchwidth is well known.
In particular, we will generally not be interested in graphs with branchwidth at most 1, as these are union of stars \cite{RobertsonS1991Graph} and can thus easily be recognised.
The following simple result allows us to reduce the problem of computing the branchwidth of a graph (and a corresponding branch-decomposition) to computing the branchwidth of its blocks.
We give an explicit proof of this result, since the explicit construction of the branch-decomposition is somewhat cumbersome and we shall reencounter ideas of the particular construction we use later on.

\begin{lemma}\label{lem_branchwidthofblocks}
    For any graph $G$ with $\bw(G) \geq 2$, the branchwidth of $G$ is the maximum of the branchwidth of its blocks.
    In particular, a branch-decomposition of $G$ with width equal to the branchwidth of $G$ can be constructed from branch-decompositions of its blocks if they each have width at most the branchwidth of $G$ in time $\Ocal(m)$.
\end{lemma}
\begin{proof}
    We first want to reduce our problem to the 2-edge-connected components of $G$, which can themselves be comprised of several blocks.
    First, we observe that if we had branch-decompositions of all components $C_1, \ldots , C_\ell$ of $G$, we can take an arbitrary ternary tree $T^c$ with $\ell$ leaves and construct a branch-decomposition of $G$ by subdividing an edge of each of the branch-decompositions of the components and identifying the subdivision vertex with one of the leaves of $T^c$.
    (We assume here that $\ell \geq 2$, since the case $\ell = 1$ is trivial.)
    One can now get a bijection between the leaves of the constructed tree and $E(G)$ by combining the bijections associated with the branch-decompositions of our components in a natural way.
    It is then easy to confirm that all edges corresponding to edges in $T^c$ have width 0 and all other edges retain their appropriate width from the existing branch-decompositions.

    Suppose now that $G$ is connected and contains a bridge $e = uv$, meaning that $G - e$ is no longer connected.
    Let $G_u$ be the component of $G - e$ that contains $u$ and let $G_v$ be defined analogously.
    (We assume that neither of these components is trivial.
    Otherwise the desired result is again trivially achieved.)
    Further assume that $G_u$ has a branch-decomposition $(T_u,\tau_u)$ with width at most the branchwidth of $G$ and let $(T_v,\tau_v)$ be a branch-decomposition of $G_v$ also with width at most the branchwidth of $G$.
    Then we construct a branch-decomposition of $G$ with width at most $\bw(G)$ as follows.
    Let $T^\star$ be the unique ternary tree with three leaves $l_1,l_2,l_3$ and the root $r$.
    Subdivide an edge of $T_u$ and identify the subdivision vertex with $l_1$, and subdivide an edge of $T_v$ and identify the subdivision vertex with $l_3$.
    Let $T^b$ be the resulting tree and let $\tau^b$ be the result of combining $\tau_u$ and $\tau_v$, as well as adding $\tau^b(l_2) = e$.
    Thus $\tau^b$ is a bijection between $E(G)$ and the leaves of $T^b$.
    Clearly the $rl_2$ in $T^b$ has width 2, justifying the requirement in the statement that we must have $\bw(G) \geq 2$.
    For $rl_1$ the only vertex which can contribute to the width here is $u$ and analogous observation holds for $rl_3$ and $v$.
    All remaining edges correspond to edges in $T_u$ or $T_v$ and retain their width from the previous branch-decompositions.

    Thus, using the constructions just presented, both of which can clearly be carried out in linear time, we can reduce the proof of the statement to the case in which $G$ is 2-edge-connected, meaning that it has no bridges and thus all of its blocks consist of more than just an edge.

    Let $B$ be a block of $G$ and let $B_1, \ldots , B_k$ be 2-edge-connected subgraphs of $G$ such that $B \cup \bigcup_{i=1}^k B_i = G$, we have $V(B_i) \cap V(B_j) = \emptyset$ for all distinct $i,j \in [k]$, and $|V(B) \cap V(B_i)| = 1$ for all $i \in [k]$, with $v_i \in V(B) \cap V(B_i)$ being a cut vertex in $G$.
    Further, suppose that $(T_B,\tau_B)$ is a branch-decomposition of $B$ and $(T_1,\tau_1), \ldots , (T_k,\tau_k)$ are branch-decompositions of $B_1, \ldots , B_k$, where each branch-decomposition has width at most $\bw(G)$.
    Since $B$ is a block, for all $i \in [k]$, we may let $e_i \in E(B_i)$ be an edge which has $v_i$ as an endpoint.
    Note that it may be true that $e_i = e_j$ for some distinct $i,j \in [k]$.

    Let $f_i \in E(T_B)$ be the unique edge incident to $\tau_B(e_i)$ in $T_B$.
    For each $i \in [k]$, we subdivide $f_i$, letting $x_i$ be the corresponding subdivision vertex.
    (This may cause $f_i$ to be subdivided more than once, which is of no concern for our later arguments.)
    Then we subdivide an arbitrary edge in $T_i$ and identify the subdivision vertex with $x_i$.
    This ultimately results in a tree $T$ with $|E(G)|$ leaves, allowing us to construct a bijection $\tau$ between $E(G)$ and $L(T)$ by simply combining $\tau_1, \ldots , \tau_k$.
    We claim that $(T,\tau)$ is a branch-decomposition of $G$ of width at most $\bw(G)$.

    In $T$ each $x_i$ is incident to three edges:
    $e_i^1$ being the edge closest to the non-leaf endpoint of $f_i$, $e_i^2$ being the edge incident to the subdivision vertex produced in $T_i$, and $e_i^3$ being the edge closest to $\tau_B(e_i)$.
    Since $v_i$ is a cut vertex in $G$, the only vertices that can contribute to width of $e_i^1$ are the endpoints of $e_i$ and thus the width of $e_i^1$ in $(T,\tau)$ is at most 2.
    For the same reason, the width of the edge $e_i^2$ can at most be 1, as $v$ is the only vertex that can appear in the cut around $E(B_i)$.
    Finally, if $e_i^3$ has $\tau_B(e_i)$ as an endpoint, its width can clearly at most be 2.
    Otherwise, it is part of a subdivided path corresponding to $f_i$ and the fact that its width is at most 2 can be derived by induction from our observation on the width of $e_i^1$.
    Any edge corresponding to an edge in $T_i$ for any $i \in [k]$ retains its width in $(T,\tau)$.
    All of the remaining edges must correspond to edges in $T_B$ and from our observations above, and in particular again the fact that $v_1, \ldots , v_k$ are cut vertices, we can deduce that they also retain their widths from $(T_B,\tau_B)$.
    Thus $(T,\tau)$ is a branch-decomposition of $G$ of width at most $\bw(G)$.
    Once more it is clear that our construction only takes linear time to carry out.

    To construct our branch-decomposition of $G$ from the branch-decompositions of its blocks, we may thus proceed as follows:
    First, combine the decompositions for all non-trivial blocks via the procedure laid out for 2-edge-connected graphs above.
    Then combine these decompositions along bridges.
    Finally, combine the decompositions of the components of $G$.
\end{proof}

More generally, it is know that there is an FPT-algorithm for the problem of computing the branchwidth of a graph.

\begin{proposition}[Bodlaender and Thilikos \cite{BodlaenderT1997Constructive}]\label{prop_fptbw}
    There exists an algorithm that, given a graph $G$, outputs $\bw(G)$ in time $2^{\Ocal((\bw(G))^3)}\cdot n$.
\end{proposition}

We call a graph \emph{planar} if it has an embedding in the sphere.
A core result for our purposes is the fact that it is possible to decide in polynomial time whether the branchwidth of planar graphs is at most some specified integer $k$.

\begin{proposition}[Seymour and Thomas \cite{SeymourT1994Call}]\label{prop_catchrat}
    Given a planar graph $G$ and a positive integer $k$, there exists an algorithm that decides in time $\Ocal(n^2)$ whether $G$ has branchwidth at most $k$.
\end{proposition}

This is remarkable, since this problem is NP-complete in general graphs \cite{SeymourT1994Call}.
In contrast, Seymour and Thomas provide a polynomial time algorithm for finding an optimal branch-decomposition for a planar graph in time $\Ocal(n^4)$ \cite{SeymourT1994Call}.
This has since been improved to the following.

\begin{proposition}[Gu and Tamaki \cite{GuT2008Optimal}]\label{prop_catchratbranchdecomp}
    Given a planar graph $G$ and a positive integer $k$, there exists an algorithm that returns a branch-decomposition of $G$ with width $\bw(G)$ in $\Ocal(n^3)$-time.
\end{proposition}

\paragraph{Tangles.}
A \emph{separation} in a graph $G$ is a pair $(A,B)$ of vertex sets such that $A \cup B = V(G)$ and there is no edge in $G$ with one endpoint in $A \setminus B$ and the other in $B \setminus A$.
The \emph{order} of $(A,B)$ is $|A\cap B|$.
If $k$ is a positive integer, we denote the collection of all separations $(A,B)$ of order less than $k$ in $G$ as $\mathcal{S}_k(G)$.

An \emph{orientation} of $\mathcal{S}_k(G)$ is a set $\mathcal{O}$ such that for all $(A,B) \in \mathcal{S}_k(G)$ exactly one of $(A,B)$ and $(B,A)$ belongs to $\mathcal{O}$. 
A \emph{tangle} of order $k$ in $G$ is an orientation $\mathcal{T}$ of $\mathcal{S}_k(G)$ such that for all $(A_1,B_1), (A_2,B_2), (A_3,B_3) \in \mathcal{T}$, it holds that $G[A_1] \cup G[A_2] \cup G[A_3] \neq G$.
If $\mathcal{T}$ is a tangle and $(A,B)\in\mathcal{T}$ we call $A$ the \emph{small side} and $B$ the \emph{big side} of $(A,B)$.

Let $G$ be a graph and $\mathcal{T}$ and $\mathcal{D}$ be tangles of $G$.
We say that $\mathcal{D}$ is a \emph{truncation} of $\mathcal{T}$ if $\mathcal{D}\subseteq\mathcal{T}$.

Tangles witness the fact that a graph has large branchwidth.
\begin{proposition}[Robertson and Seymour \cite{RobertsonS1991Graph}]\label{prop_tanglebranchwidthduality}
    Let $G$ be a graph and let $k$ be an integer with $k \geq 2$.
    Then $G$ has a tangle of order $k$ if and only if $G$ has branchwidth at least $k$. 
\end{proposition}

If we are dealing with graph with a 2-cell-embedding in a surface of positive genus with high representativity, then we are guaranteed to find a large tangle and thus know that the graph we embedded has high branchwidth.
We note that the version of this theorem we state here is less specific than what Robertson and Seymour actually prove (see (4.1) in \cite{RobertsonS1994Graph}).

\begin{proposition}[Robertson and Seymour \cite{RobertsonS1994Graph}]\label{prop_highrepresentativitygivestangle}
    Let $k$ be a positive integer and let $G$ be a graph with a 2-cell-embedding in a surface of positive genus with representativity $k$.
    Then $G$ has a tangle of order $k$.
\end{proposition}

\paragraph{Minors.}
We formally introduce minors to later discuss how to find grid-minor models in case the branchwidth of the graphs we consider turns out to be too high.
A function $\varphi \colon V(H) \rightarrow 2^{V(G)}$ is called a \emph{minor model (of $H$ in $G$)} (or simply a \emph{model}) if $\varphi(V(H))$ is a set of pairwise disjoint vertex sets, each inducing a connected subgraph of $G$, and for each edge $uv \in E(H)$ there exists an edge $ab \in E(G)$ with $a \in \varphi(u)$ and $b \in \varphi(v)$.
A graph $H$ is a \emph{minor} of a graph $G$ if and only if there exists a minor model of $H$ in $G$.
This also causes us to say that $G$ has an \emph{$H$-minor}.

\paragraph{Treewidth and tree-decompositions.}
Since we will need it to apply some external results, we also give a formal definition of treewidth, which is asymptotically equivalent to branchwidth and generally more popular, but only ancillary to our approach.
A \emph{tree-decomposition} of a graph $G$ is a pair $\mathcal{T} = (T, \beta)$, where $T$ is a tree and $\beta \colon V(T) \rightarrow 2^{V(G)}$ such that
\begin{itemize}
    \item for every edge $e \in E(G)$, there exists a $t \in V(T)$ such that $e \subseteq \beta(t)$, and
    \item for every vertex $v \in V(G)$, the set $\beta^{-1}(v)$ induces a connected, non-empty subtree of $T$.
\end{itemize}
The \emph{width} of $\mathcal{T}$ is the maximum value of $|\beta(t)| - 1$ over all $t \in V(T)$ and the \emph{treewidth} $\mathsf{tw}(G)$ of $G$ is the minimum width over all tree-decompositions of $G$.

For each $t \in V(T)$, we define the \emph{adhesion sets of $t$} as the sets in $\{\beta(t) \cap \beta(d) ~\!\colon\!~ d \in N_T(t)\}$, and the maximum size of them is called the \emph{adhesion of $t$}.
The \emph{adhesion} of $\mathcal{T}$ is the maximum adhesion of a node of $\mathcal{T}$.\footnote{For tree-decompositions with a single vertex we define the adhesion to be 0.}
If $T$ has a designated root $r$, we call $(T,r,\beta)$ a \emph{rooted} tree-decomposition.

As mentioned in the introduction, branchwidth is asymptotically equivalent to branchwidth.

\begin{proposition}[Robertson and Seymour \cite{RobertsonS1991Graph}]\label{prop:twbwequiv}
    Let $G$ be a non-acyclic graph.
    Then we have $\bw(G)-1 \leq \tw(G) \leq \lfloor \nicefrac{3}{2}\cdot \bw(G)\rfloor-1$.
\end{proposition}

\section{Approximating the branchwidth of graphs of bounded Euler genus}\label{sec:EulerGenus}
The core result of this section is that, if we are given some integer $k$, we can efficiently confirm that either the branchwidth of a graph with a 2-cell-embedding in some given surface of genus $g$ is at most $g(k-1) + k$ or the branchwidth of said graph is at least $k$.
In particular, at the cost of a worse runtime and slightly worse bounds on our approximation, we can even efficiently provide a branch-decomposition, if the branchwidth is low, or a grid-minor model as a witness that the branchwidth is above some specified bound.

\subsection{Dealing with apices}
We start with a simple lemma that will remain useful in the later sections of this article.

\begin{lemma}\label{lem_branchdecompositionviacutset}
    Let $c$ be an integer, let $G$ be a graph with $\bw(G) \geq 2$, and let $X \subseteq V(G)$ be a set of vertices with $|X| \leq c$.
    Then $\bw(G) \leq \bw(G - X) + c$.
    In particular, if we are provided with branch-decompositions of width at most $\bw(G - X)$ for each of the blocks of $G - X$, we can construct a branch-decomposition of $G$ of width at most $\bw(G - X) + c$ in time $\Ocal(c m)$.
\end{lemma}
\begin{proof}
    Let $b \coloneqq \bw(G - X)$.
    Then we can derive a branch-decomposition $(T',\tau')$ of $G - X$ width at most $b$ in linear time via \zcref{lem_branchwidthofblocks}.
    We may assume that $X$ is non-empty and let $E_X$ be the set of edges in $E(G)$ that have at least one endpoint in $X$.
    
    Let $k \coloneqq |V(G - X)|$, let $z_1, \ldots , z_k$ be the vertices of $G - X$, and let $E_0,E_1, \ldots , E_k \subseteq E_X$ be a partitioning of $E_X$ into (possibly empty) sets, such that for all $i \in [k]$, the edge $e \in E_X$ is in $E_i$ if $e = z_ix$ for some $x \in X$ and thus $E_0 = E_X \setminus \bigcup_{i=1}^k E_i$.
    We assume that the vertices of $G - X$ are labelled such that there exists a $k' \in [k]$ with the property that for any $i \in [k']$ there exists an edge $e_i \in E(G - X)$ that is incident to $z_i$ and for any $i \in [k'+1,k]$, the vertex $z_i$ is isolated in $G - X$.
    For each $i \in [k']$, we let $(\widehat{T}_i,\widehat{\tau}_i)$ be defined such that $\widehat{T}_i$ is an arbitrary ternary tree with $|E_i|$ leaves and $\widehat{\tau}_i$ is a bijection from $E_i$ to $L(\widehat{T}_i)$.
    We let $(\widehat{T}_0,\widehat{\tau}_0)$ be the pair of a ternary tree $\widehat{T}_0$ with $| E_0 \cup \bigcup_{i=k'+1}^k E_i |$ leaves\footnote{Here and for the remainder of the proof, we ignore the possibility that $| E_0 \cup \bigcup_{i=k'+1}^k E_i | = 0$. This is safe as our constructions could easily be simplified in this specific case.} and a bijection $\widehat{\tau}_0$ from $E_0 \cup \bigcup_{i=k'+1}^k E_i$ to $L(\widehat{T}_0)$.
    
    This allows us to modify the branch-decomposition $(T',\tau')$ of $G - X$ as follows.
    First, we subdivide an arbitrary edge of $\widehat{T}_0$, letting $w_0$ be the subdivision vertex and identify $w_0$ with the subdivision vertex resulting from subdividing an arbitrary edge in $T'$.
    We call the resulting tree $T''$.
    Let $u_i \coloneqq \tau'(e_i)$ for all $i \in [k']$.
    For each $i \in [k']$ we first subdivide an arbitrary edge in $\widehat{T}_i$\footnote{We gloss over the case in which $T_i$ consists of a single vertex. This makes our construction simpler and does not influence the truth of our statement.}, letting $w_i$ be the subdivision vertex.
    Then, we subdivide the unique edge of $T''$ that is incident to $u_i$, letting $w_i'$ be the subdivision vertex, and finish by identifying $w_i$ and $w_i'$.
    Let $T$ be the resulting tree.
    We extend $\tau'$ and the other bijections to $T$ in the obvious way, meaning $\tau$ is a bijection from $E(G) = E_X \cup E(G - X)$ to $L(T)$ such that $\tau(e)$ is defined as $\tau'(e)$, if $e \in E(G - X)$, as $\widehat{\tau}_0(e)$, if $e \in E_0 \cup \bigcup_{i=k'+1}^k E_i$, and as $\widehat{\tau}_i(e)$, if $e \in E_i$ for some $i \in [k']$.

    Every edge $e \in E(T)$ now naturally defines a bipartition $X_e,Y_e$ of the edges of $G$ via $\tau$.
    It is easy to see that for any edge $e \in E(T)$ that corresponds to an edge of $\bigcup_{i=0}^{k'} E(\widehat{T}_i)$, we have $\partial_G(X_e) \leq c$.
    Consider an edge $e \in E(T)$ that corresponds to an edge of $T'$.
    Then this edge in particular induces a bipartition $X_e',Y_e'$ of the edges in $G - X$ via $T'$ and $\tau'$.
    Via our assumption, we know that $\partial_{G - X}(X_e') \leq \bw(G - X)$.
    
    Suppose that there exists a $v \in \partial_G(X_e) \setminus ( \partial_{G - X}(X_e') \cup X )$.
    This implies, w.l.o.g.\ that there exists some $f \in E(G - X) \cap X_e$ and $g \in E_X \cap Y_e$ such that $f$ and $g$ are both incident to $v$.
    We first note that $g \not\in E_0 \cup \bigcup_{i=k'+1}^k E_i$ thanks to the definition of $E_0$ and our choice of $k'$.
    Thus $g \in E_i$ for some $i \in [k']$.
    Next, we note that we may assume that $e$ is not incident to $u_i$ for any $i \in [k']$, as this would imply that $|\partial_G(X_e)| \leq 2$, which is acceptable thanks to our assumption that $\bw(G) \geq 2$.
    Together this implies that $e_i,g \in Y_e$.
    However, this means that $v \in \partial_{G - X}(X_e')$, contradicting our initial choice of $v$.
    This confirms that we have $\partial_G(X_e) \leq \bw(G - X) + c$ for every edge $e \in E(T)$.

    To see that this procedure runs in linear time in the number of edges of $G$, we first note that the number of blocks of $G$ is bounded by $|V(G)|$ and we perform a constant number of operations for each of the blocks to construct the initial tree $T'$ and $\tau'$.
    Since we give no explicit relation between the size of $X$ and $G$, we can only give the estimate from above on $|E(X)|$ as $|E(G)|$, which ultimately dominates the runtime of our procedure, as we need to place each edge in $E(G)$ into the branch-decomposition.
    Finally, note that we include $c$ in the runtime estimate only for the technical possibility that $c \geq m$, implying that $G$ is a forest.
\end{proof}

\subsection{Approximating the branchwidth of a graph of bounded genus}
We will also need a pair of generalizations of a result on the existence of grids in planar graphs found in \cite{RobertsonST1994Quickly} for the setting of graphs embedded in surfaces of positive genus.
For this purpose, let us first properly introduce what we mean by a grid.

Let $n,m \in \mathbb{N}$ be two positive integers.
The \emph{$(n \times m)$-grid} is the graph $G$ with the vertex set $V(G) = [n] \times [m]$ and the edges
\begin{align*}
E(G) = \big\{ \{ (i, j) , (\ell , k) \} ~\!\colon\!~    & i, \ell \in [n], \ j,k \in [m], \text{ and } \ \\
                                                        & ( |i - \ell| = 1  \text{ and } j = k ) \text{ or } ( |j - k| = 1 \text{ and } i = \ell ) \big\} .
\end{align*}
An \emph{$n$-grid} is an $(n \times n)$-grid.
A central result from the Graph Minors Series tells us that minor models of large grids are witnesses of the fact that a graph has high branchwidth.
In particular, there exists a function $h \colon \mathbb{N} \rightarrow \mathbb{N}$ such that for every positive integer $k \in \mathbb{N}$, every graph $G$ with $\tw(G) \geq h(k)$ contains a $k$-grid as a minor \cite{RobertsonS1986Grapha}.
Currently, the best known bound for the function $h(k)$ is $\Ocal(k^9\cdot \polylog(k))$ \cite{ChuzhoyT2021Tighter}.
However, for graphs embeddable on surfaces, better bounds are known, beginning with a classic linear bound for planar graphs.

\begin{proposition}[Robertson, Seymour, and Thomas \cite{RobertsonST1994Quickly} (see (6.3))]\label{thm_planargrid}
    Let $r$ be a positive integer and let $G$ be a planar graph.
    If $\bw(G) \geq 4r$ then $G$ contains an $r$-grid minor.
\end{proposition}

We will use a generalised version of this result for graphs of bounded genus, which we present together with the core tool used in its proof that discusses representativity.

\begin{proposition}[Demaine, Fomin, Hajiaghayi, and Thilikos \cite{DemaineFHT2005Subexponential} (see Lemma 3.3)]\label{prop_gimmedagridfacewidth}
    Let $r$ be a positive integer and let $G$ be a graph with a 2-cell embedding in a surface of positive genus with representativity at least $4r$.
    Then $G$ contains an $r$-grid as a minor.
\end{proposition}

\begin{proposition}[Demaine, Fomin, Hajiaghayi, and Thilikos \cite{DemaineFHT2005Subexponential} (see Theorem 4.12)]\label{prop_gimmedagridbw}
    Let $r$ be a positive integer and let $G$ be a graph with Euler genus $g$ and $\bw(G) > 4r(g+1)$.
    Then $G$ contains an $r$-grid as a minor.
\end{proposition}

The proof of the main theorem of this section features one of the main tools that is used throughout this paper.
Namely, to reach graphs on which we can either confirm that the branchwidth of $G$ is too high or which we can find branch-decompositions for, we simply cut apart the embedded graph via short non-contractible curves.
This technique will be expanded on later, where we will need to employ a few more tricks to make it work, but the core idea remains the same.

\begin{theorem}\label{thm_boundedgenusrats}
    Let $k,g$ be integers and let $G$ be a graph with a 2-cell-embedding $\psi$ in a surface $\Sigma$ of genus $g$.
    There exists an algorithm that determines whether $G$ has branchwidth at least $k$ or branchwidth at most $g(k-1) + k$ in time $\Ocal(g^2k(n+m)+n^2)$.
    
    In particular, in the second outcome, the algorithm returns a branch-decomposition of $G$ with width at most $g(k-1) + k$ if we allow for a $\Ocal(g^2k(n+m)+n^3)$-runtime.
\end{theorem}
\begin{proof}
    We may assume that $\Sigma$ is not the sphere, as we can otherwise simply apply \zcref{prop_catchrat} or \zcref{prop_catchratbranchdecomp}, depending on whether we want a branch-decomposition or not.
    First, we apply \zcref{prop_findnoncontrcurve} to find a non-contractible curve $\gamma$ witnessing the representativity of $\psi$, which only intersects $\psi$ in the vertices of $G$.
    If $\gamma$ intersects at least $k$ vertices, then $\psi$ has representativity at least $k$ and thus, according to \zcref{prop_highrepresentativitygivestangle}, $G$ contains a tangle of order $k$, which further implies that $G$ has branchwidth at least $k$, due to \zcref{prop_tanglebranchwidthduality}.
    Thus we may instead assume that $\gamma$ intersects less than $k$ vertices of $G$.

    Let $V(\gamma)$ be the vertices that $\gamma$ intersects.
    If we delete $V(\gamma)$ from $G$, this results in two subgraphs $G_1,G_2$ of $G$, with $G_1 \cup G_2 = G - V(\gamma)$, each embedded in one of the two components $\Sigma_1, \Sigma_2$ of $\Sigma - \gamma$.
    Since $\gamma$ is non-contractible, the sum of the genus of $\Sigma_1$ and the genus of $\Sigma_2$ is less than the genus of $\Sigma$.
    For both $i \in [2]$, each block of $G_i$ has a 2-cell-embedding in $\Sigma_i$ induced by $\psi$.

    We may now repeat this procedure on $G_1$ and $G_2$, which will terminate as each iteration reduces the genus of the surfaces involved.
    If we find a subgraph of $G$ embedded on a surface of positive genus with representativity at least $k$, we known that the branchwidth of $G$ is at least $k$ according to \zcref{prop_highrepresentativitygivestangle,prop_tanglebranchwidthduality}.
    Thus we may assume that, after deleting at most $g(k - 1)$ vertices of $G$, we have arrived at a subgraph $G' \subseteq G$ with $|V(G')| + g(k-1) \leq |V(G)|$ such that each block of $G'$ is planar.
    This allows us to use \zcref{prop_catchrat} on each block to either confirm that it, and thus $G$, has branchwidth at least $k$ in $\Ocal(|V(G)|^2)$-time, or even find a branch-decomposition of the block with width at most $k$ in $\Ocal(|V(G)|^3)$-time via \zcref{prop_catchratbranchdecomp}.
    To extend this to branch-decomposition of $G$ with width at most $g(k-1) + k$ we simply apply \zcref{lem_branchdecompositionviacutset}.
\end{proof}

\zcref{prop_findembedding} and \zcref{thm_boundedgenusrats} allow us to approximate the branchwidth of a bounded genus graph.

\begin{corollary}\label{cor_boundedgenusfindbw}
    Let $k,g$ be integers and let $G$ be a graph of Euler genus at most $g$.
    There exists a computable function $f \colon \mathbb{N} \to \mathbb{N}$ and an algorithm that determines whether $G$ has branchwidth at least $k$ or branchwidth at most $g(k-1) + k$ in time $\Ocal(f(g)k(n+m)+n^2)$.
    
    In particular, in the second outcome, the algorithm returns a branch-decomposition of $G$ with width at most $g(k-1) + k$ if we allow for a $\Ocal(f(g)k(n+m)+n^3)$-runtime.
\end{corollary}

Via binary search, this further allows us to approximate the branchwidth of our graph of bounded genus.
Note that we only need to construct the branch-decomposition once we are actually sure that the branchwidth of our graph falls below our desired bound, which means that we only need to compute the branch-decomposition in the last instance of our search process.

\begin{corollary}\label{cor_boundedgenusfindbd}
    Let $G$ be a graph of Euler genus at most $g$.
    There exists a computable function $f \colon \mathbb{N} \to \mathbb{N}$ and an algorithm that outputs an integer $b$ such that $b \leq \bw(G) \leq g(b-1) + b$ in time $\Ocal((f(g)(n+m)+n^2) \log n)$ and furthermore outputs a branch-decomposition of width at most $g(b-1)+b$ in time $\Ocal((f(g)(n+m)) \log n + n^3)$.
\end{corollary}

\subsection{Constructing a grid}\label{subsec:findgridbdgenus}
In our algorithms a possible answer is that the branchwidth of the input graph is lower bounded by some (approximate) value.
In this section, we explain how this lower bound can be (approximately) certified by some minor-model of a grid, in polynomial time.
For this purpose, we employ some known results powerful enough to do most of the work for us.

Let $G$ be a graph with a 2-cell-embedding of $G$ in a surface $\Sigma$ of Euler genus $g$.
The \emph{radial} graph $R_{G}$ of $G$  is a bipartite graph whose vertices are the vertices and the faces of $G$ and where adjacency expresses incidence between the corresponding faces and vertices.
The definition of $R_{G}$ naturally defines an embedding of $R_{G}$ in $\Sigma$. 
A path in $R_{G}$ defines a \emph{radial path} in $G$ between the corresponding vertices or faces.
Accordingly, the \emph{distance} in $R_{G}$ between two vertices defines the \emph{radial distance} in $G$ between the corresponding vertices or faces.
Given a subgraph $H$ of the radial graph $R_{G}$, we denote by $N(H)$ the set of all vertices of $R_{G}$ whose distance from some vertex of $V(G)\cap V(H)$ is at most two.   
Also we denote by $E_{H}(G)$ the set of all edges of $G$ where both their endpoints belong in $N(H)$.
We denote by $G/E_{H}(G)$ the graph obtained by $G$ if we contract all edges of $G$.
We make use of the following result from \cite{GolovachKST25FindingARXIV,GolovachKST25Finding} that is, in turn, based on the main result of \cite{DemaineHK11Contraction} and the algorithmic results of \cite{CabelloCL2012Algorithms}.

\begin{proposition}[Golovach, Kolliopoulos, Stamoulis, and Thilikos \mbox{\cite[Lemma 7.5]{GolovachKST25FindingARXIV}}]
\label{lem_trans_rad}
    There is an algorithm that, given a  2-cell-embedding of a graph $G$ in a surface $\Sigma$ of Euler genus $g$ outputs, in linear time, a subgraph $H$ of the radial graph of $G$ such that the following hold:
\begin{itemize}
    \item $V(H)$ is the union of the vertex sets of a collection $\Pcal$ of $\mathcal{O}(g)$ many shortest radial paths in $G$,
	\item $E_{H}(G)$ is connected in $G$,
	\item $G' \coloneqq G/E_{H}(G)$ is a planar graph, and
	\item $\tw(G)=\mathcal{O}(\eg(G)^{\nicefrac{5}{2}} \cdot \tw(G'))$.
\end{itemize}
\end{proposition}

We also need the following result by Gu and  Tamaki in \cite{GuT11Constant}.
\begin{proposition}[Gu and Tamaki \cite{GuT11Constant}]
\label{prop_aprox_alg}
For every integer $c ≥ 1$ and every $b > 2c + \nicefrac{3}{2}$, there is an algorithm that, given a planar graph $G$,  constructs the minor model $\varphi \colon V(\Gamma_{k}) \rightarrow 2^{V(G)}$ of a $k$-grid $\Gamma_{k}$
with $k ≥ \bg(G)/b$ in  $\mathcal{O}(n^{1+ \frac{1}{c}}\cdot  \log n)$ time.
\end{proposition}

Combined these results allow us to show that we can in fact find a grid minor model in a graph embedded on a surface of bounded genus if the treewidth of the graph is high enough.

\begin{lemma}\label{lem_make_con_genus}
There exist a constant $c_{\ref{lem_make_con_genus}}$ and an algorithm running in $\Ocal(n^2)$-time that, given a 2-cell-embedding of a graph $G$ in a surface $\Sigma$ of Euler genus $g$, where $\tw(G) \geq c_{\ref{lem_make_con_genus}} \cdot g^{\nicefrac{5}{2}} \cdot k$, return a minor model $\varphi$ of a $k$-grid $\Gamma_{k}$ such that $\bigcup_{v\in \Gamma_{k}}G[\varphi(v)]$ is drawn in some disk within $\Sigma$.
\end{lemma}
\begin{proof}
Let $c'$ be a constant such that for every planar graph $G$, we have $\tw(G) \leq c' \cdot \bg(G)$ (e.g. we may pick $c'=5$ \cite{Grigoriev11Tree}).
Choose $c_{\ref{lem_make_con_genus}}$  so that 
we may  apply the algorithm  in \zcref{lem_trans_rad}
to construct a planar graph $G' \coloneqq G/E_{H}(G)$
where $\tw(G') \geq 16\cdot c' \cdot k$, which implies $\bg(G') \geq 16k$.
Let $x$ be the vertex of $G'$ that is the result of the contraction of the edges in $E_H(G)$.

Then we apply the algorithm of \zcref{prop_aprox_alg} for $c=10/9$ and $b=4$ and obtain the minor model $\varphi \colon V(\Gamma_{4k}) \rightarrow 2^{V(G')}$ of a $4k$-grid $\Gamma_{4k}$  in  $\mathcal{O}(n^{1.9}\cdot  \log n)=\mathcal{O}(n^2)$ time. 
Notice that $\varphi$ is also a minor model of the union of $4$
disjoint copies of the  $k$-grid $\Gamma_{k}$.
This gives rise to four minor-models $\varphi_{1},\ldots,\varphi_{4}$
of $\Gamma_{k}$ whose unions of images are pairwise disjoint.
Moreover, in the planar embedding of $G'$ we may further assume, for each $\varphi_{i}$, with $i\in[4]$, that
all images of $\varphi_{i}$ can be drawn in a disk $\Delta_{i}$
such that the disks $\Delta_{1},\ldots,\Delta_{4}$ are pairwise disjoint. Let $\Delta_{i}$ be one of these disks where $x$ is not drawn. Notice now that $\Delta_{i}$ can be seen as a disk of the surface $\Sigma$ where $G$ is drawn because un-contracting the edges of $E_{H}(G)$ does not affect the  $x$-avoiding disk  $\Delta_{i}$.
Therefore  $\varphi_{i}$ can be returned as the claimed minor model of a $k$-grid in $G$.
\end{proof}

Notice that according to \zcref{prop_aprox_alg}
the running time of \zcref{lem_make_con_genus}
can be pushed arbitrarily close to linear in the cost of worst choices 
of the constant $c_{\ref{lem_make_con_genus}}$.

Finally, note that \zcref{lem_make_con_genus} allows us to refine \zcref{thm_boundedgenusrats,cor_boundedgenusfindbw}, as we can apply this theorem in the proof of \zcref{thm_boundedgenusrats} by demanding higher representativity.
Using \zcref{prop_tanglebranchwidthduality,prop_highrepresentativitygivestangle,prop:twbwequiv} this then gives us the necessary lower bound on the treewidth of our graph to apply \zcref{lem_make_con_genus} to find the desired grid minor model.
However, this comes at the cost of driving up the bounds for the branchwidth in our approximation.

\begin{theorem}
    Let $k,g$ be integers and let $G$ be a graph of Euler genus at most $g$.
    There exists a computable function $f \colon \mathbb{N} \to \mathbb{N}$ and an algorithm that determines whether $G$ has branchwidth at least $k$ or branchwidth at most $\nicefrac{3}{2}g(c_{\ref{lem_make_con_genus}}\cdot g^{\nicefrac{5}{2}} \cdot k) + k$ in time $\Ocal(f(g)k(n+m)+n^2)$.
    
    In particular, in the first outcome, the algorithm returns a $k$-grid-minor model and in the second outcome, the algorithm returns a branch-decomposition of $G$ with width at most $\nicefrac{3}{2}g(c_{\ref{lem_make_con_genus}}\cdot g^{\nicefrac{5}{2}} \cdot k) + k$ if we allow for a $\Ocal(f(g)k(n+m)+n^3)$-runtime.
\end{theorem}

\section{Graph Minor Structure Theory}\label{sec_gmst}
As we get towards the more involved part of our article, we dive into the details of the structure theory introduced by Robertson and Seymour.
Our goal will be to reprove the Graph Minor Structure Theorem (GMST), which first appears in \cite{RobertsonS2003Grapha}, to ensure an additional helpful property.
In particular, we desire more information on how exactly each part of the decomposition it provides us can be ``painted'' onto a surface.
This part mainly requires us to examine the existing proof of this theorem carefully and add one ingredient from \cite{Grohe2016Quasi4Connected}.
We note that we only need to reprove the very last step of the GMST, as we can directly make use of the Local Structure Theorem.
Thus, aside from having to introduce several definitions, we do not have to invest a lot of work to accomplish our goal.

\subsection{Underlying definitions of graph minor structure theory}
The definitions we present are inspired by the work of \cite{KawarabayashiTW2018New,KawarabayashiTW2021Quickly,GorskySW2025Polynomial}.
Of course, most of them are in some way descendants of the work of Robertson and Seymour.

\paragraph{Meshes.}
Let $A, B \subseteq V(G)$, an \emph{$A$-$B$-path} is a path $P$ that has one endpoint in $A$, the other in $B$, and $V(P) \cap (A \cup B)$ only contains the endpoints of $P$.
Let $n,m$ be integers with $n,m\geq 2$.
A \emph{$(n\times m)$-mesh} is a graph $M$ which is the union of paths $M=P_1\cup\cdots\cup P_n\cup Q_1\cup \cdots \cup Q_m$ where
    \begin{itemize}
        \item $P_1,\cdots,P_n$ are pairwise vertex-disjoint, and $Q_1,\cdots,Q_m$ are pairwise vertex-disjoint.
        \item for every $i\in [n]$ and $j\in [m]$, the intersection $P_i\cap Q_j$ induces a path,
        \item each $P_i$ is a  $V(Q_1)$-$V(Q_m)$-path intersecting the paths $Q_1,\cdots Q_m$ in the given order, and each $Q_j$ is a $V(P_1)$-$V(P_m)$-path intersecting the paths $P_1,\cdots, P_h$ in the given order. 
    \end{itemize}
We say that the paths $P_1,\cdots,P_n$ are the \emph{horizontal paths}, and the paths $Q_1,\cdots,Q_m$ are the \emph{vertical paths}.
A mesh $M'$ is a \emph{submesh} of a mesh $M$ if every horizontal (vertical) paths of $M'$ is a subpath of a horizontal (vertical) paths $M$, respectively.
We write \emph{$n$-mesh} as a shorthand for an $(n \times n)$-mesh.

\paragraph{More on tangles.}
Let $r \in \mathbb{N}$ with $r\geq 3$, let $G$ be a graph, and $M$ be an $r$-mesh in $G$.
Let $\mathcal{T}_M$ be the orientation of $\mathcal{S}_r$ such that for every $(A,B)\in\mathcal{T}_M$, the set $B\setminus A$ contains the vertex set of both a horizontal and a vertical path of $M$, we call $B$ the \emph{$M$-majority side} of $(A,B)$.

\paragraph{Paintings in surfaces.}
A \emph{painting} in a surface $\Sigma$ is a pair $\Gamma = (U,N)$, where $N \subseteq U \subseteq \Sigma$, $N$ is finite, $U \setminus N$ has a finite number of arcwise-connected components, called \emph{cells} of $\Gamma$, and for every cell $c$, the closure $\overline{c}$ is a closed disk where $N_\Gamma(c) \coloneqq \overline{c} \cap N \subseteq \mathsf{bd}(\overline{c})$.
If $|N_\Gamma(c)| \geq 4$, the cell $c$ is called a \emph{vortex}.
We further let $N(\Gamma) \coloneqq N$, let $U(\Gamma) \coloneqq U$, and let $C(\Gamma)$ be the set of all cells of $\Gamma$.
\medskip

Any given painting $\Gamma = (U,N)$ defines a hypergraph with $N$ as its vertices and the set of closures of the cells of $\Gamma$ as its edges.
Accordingly, we call $N$ the \emph{nodes} of $\Gamma$.

\paragraph{$\Sigma$-renditions.}
Let $G$ be a graph and $\Sigma$ be a surface.
A \emph{$\Sigma$-rendition} of $G$ is a triple $\rho = (\Gamma, \sigma, \pi)$, where
\begin{itemize}
    \item $\Gamma$ is a painting in $\Sigma$,
    \item for each cell $c \in C(\Gamma)$, $\sigma(c)$ is a subgraph of $G$, and
    \item $\pi \colon N(\Gamma) \to V(G)$ is an injection,
\end{itemize}
such that
\begin{description}
    \item[R1] $G = \bigcup_{c \in C(\Gamma)}\sigma(c)$,
    \item[R2] for all distinct $c,c' \in C(\Gamma)$, the graphs $\sigma(c)$ and $\sigma(c')$ are edge-disjoint,
    \item[R3] $\pi(N_\Gamma(c)) \subseteq V(\sigma(c))$ for every cell $c \in C(\Gamma)$, and
    \item[R4] for every cell $c \in C(\Gamma)$, we have $V(\sigma(c) \cap \bigcup_{c' \in C(\Gamma) \setminus \{ c \}} (\sigma(c'))) \subseteq \pi(M_\Gamma(c))$.
\end{description}
We write $N(\rho)$ for the set $N(\Gamma)$, let $N_\rho(c) = N_\Gamma(c)$ for all $c \in C(\Gamma)$, and similarly, we lift the set of cells from $C(\Gamma)$ to $C(\rho)$.
If it is clear from the context which $\rho$ is meant, we will sometimes simply write $N(c)$ instead of $N_\rho(c)$, and if the $\Sigma$-rendition $\rho$ for $G$ is understood from the context, we usually identify the sets $\pi(N(\rho))$ and $N(\rho)$ along $\pi$ for ease of notation.

\paragraph{Societies.}
Let $\Omega$ be a cyclic ordering of the elements of some set which we denote by $V(\Omega)$.
A \emph{society} is a pair $(G,\Omega)$, where $G$ is a graph and $\Omega$ is a cyclic ordering with $V(\Omega)\subseteq V(G)$.
For a given set $S \subseteq V(\Omega)$ a vertex $s \in S$ is an \emph{endpoint} of $S$ if there exists a vertex $t \in V(\Omega) \setminus S$ that immediately precedes or succeeds $s$ in $\Omega$.
We call $S$ a \emph{segment} of $\Omega$ if $S$ has two or less endpoints.

Let $(G,\Omega)$ be a society and let $\Sigma$ be a surface with one boundary component $B$ homeomorphic to the unit circle.
A \emph{rendition} of $(G,\Omega)$ in $\Sigma$ is a $\Sigma$-rendition $\rho$ of $G$ such that the image under $\pi_{\rho}$ of $N(\rho) \cap B$ is $V(\Omega)$ and $\Omega$ is one of the two cyclic orderings of $V(\Omega)$ defined by the way the points of $\pi_{\rho}(V(\Omega))$ are arranged in the boundary $B$.

\paragraph{Traces of paths and cycles.}
Let $\rho$ be a $\Sigma$-rendition of a graph $G$.
For every cell $c \in C(\rho)$ with $|N_\rho(c)| = 2$, we select one of the components of $\mathsf{bd}(c) - N_\rho(c)$.
This selection will be called a \emph{tie-breaker in $\rho$}, and we assume that every rendition comes equipped with a tie-breaker.

Let $G$ be a graph and $\rho$ be a $\Sigma$-rendition of $G$.
Let $Q$ be a cycle or path in $G$ that uses no edge of $\sigma(c)$ for every vortex $c \in C(\rho)$.
We say that $Q$ is \emph{grounded} if it uses edges of $\sigma(c_1)$ and $\sigma(c_2)$ for two distinct cells $c_1, c_2 \in C(\rho)$, or $Q$ is a path with both endpoints in $N(\rho)$.
If $Q$ is grounded we define the \emph{trace} of $Q$ as follows.
Let $P_1,\dots,P_k$ be distinct maximal subpaths of $Q$ such that $P_i$ is a subgraph of $\sigma(c)$ for some cell $c$.
Fix $i \in [k]$.
The maximality of $P_i$ implies that its endpoints are $\pi(n_1)$ and $\pi(n_2)$ for distinct nodes $n_1,n_2 \in N(\rho)$.
If $|N_\rho(c)| = 2$, let $L_i$ be the component of $\mathsf{bd}(c) - \{ n_1,n_2 \}$ selected by the tie-breaker, and if $|N_\rho(c)| = 3$, let $L_i$ be the component of $\mathsf{bd}(c) - \{ n_1,n_2 \}$ that is disjoint from $N_\rho(c)$.
We define $L_i'$ by pushing $L_i$ slightly so that it is disjoint from all cells in $C(\rho)$, while maintaining that the resulting curves intersect only at a common endpoint.
The \emph{trace} of $Q$ is defined to be $\bigcup_{i\in[k]} L_i'$.
If $Q$ is a cycle, its trace thus the homeomorphic image of the unit circle, and otherwise, it is an arc in $\Sigma$ with both endpoints in $N(\rho)$.

\paragraph{Aligned disks and grounded subgraphs.}
Let $G$ be a graph and let $\rho = (\Gamma, \sigma, \pi)$ be a $\Sigma$-rendition of $G$. 
We say that a 2-connected subgraph $H$ of $G$ is \emph{grounded (in $\rho$)} if every cycle in $H$ is grounded and no vertex of $H$ is drawn by $\Gamma$ in a vortex of $\rho$.
A disk in $\Sigma$ is called \emph{$\rho$-aligned} if its boundary only intersects $\Gamma$ in nodes.
If $H$ is planar, we say that it is \emph{flat in $\rho$} if there exists a $\rho$-aligned disk $\Delta \subseteq \Sigma$ which contains all cells $c \in C(\rho)$ with $E(\sigma(c)) \cap E(H) \neq \emptyset$ and $\Delta$ does not contain any vortices of $\Gamma$.

\paragraph{Linear decompositions of vortices.}
Let $(G,\Omega)$ be a society.
A \emph{linear decomposition} of $(G,\Omega)$ is a labelling $v_1,v_2,\dots,v_n$ of $V(\Omega)$ such that $v_1,v_2,\dots,v_n$ appear in $\Omega$ in the order listed, together with sets $(X_1,X_2,\dots,X_n)$ such that
\begin{enumerate}
    \item $X_i\subseteq V(G)$ and $v_i\in X_i$ for all $i\in[n]$,
    \item $\bigcup_{i\in[n]}X_i=V(G)$ and for every $uv\in E(G)$ there exists $i\in[n]$ such that $u,v\in X_i$, and
    \item for every $x\in V(G)$ the set $\{ i\in[n] ~\!\colon\!~ x\in X_i \}$ forms an interval in $[n]$.
\end{enumerate}
The \emph{adhesion} of a linear decomposition is $\max \{ |X_i\cap X_{i+1}| ~\!\colon\!~ i\in[n-1] \}$.
The \emph{width} of a linear decomposition is $\max \{ |X_i| ~\!\colon\!~ i\in[n] \}$.
We say that $(G,\Omega)$ has \emph{depth} at most $k$ if there exists a linear decomposition with adhesion at most $k$ for $(G,\Omega)$.

\paragraph{Depth of vortices.}
Let $G$ be a graph and $\rho$ be a $\Sigma$-rendition of $G$ with a vortex cell $c_0$.
Notice that $c_0$ defines a society $(\sigma(c_0),\Omega_{c_0})$, where $V(\Omega_{c_0})$ is the set of vertices of $G$ corresponding $N_\rho(c_0)$.
The ordering $\Omega_{c_{0}}$ is obtained by traversing along the boundary of the closure of $c_0$ in anti-clockwise direction.
We call $(\sigma(c_0),\Omega_{c_0})$ as obtained above the \emph{vortex society} of $c_0$.

The \emph{depth} of the vortex $c_0$ is thereby defined as the depth of its vortex society.
Given a $\Sigma$-rendition $\rho$ with vortices, we define the \emph{depth of $\rho$} as the maximum depth of its vortex societies.

\subsection{The Local Structure Theorem}
We now present a version of the local structure theorem derived from \cite{GorskySW2025Polynomial}.
As noted earlier, we are leaving out a lot of additional detail that the results from \cite{GorskySW2025Polynomial} yield.
Let us first define one more notion on the relation of meshes and tangles.

\paragraph{$M$-centrality.}
Let $\Sigma$ be a surface and let $\rho$ be a $\Sigma$-rendition of a graph $G$ containing an $r$-mesh $M$.
We say that $\rho$ is \emph{$M$-central} if there is no cell $c \in C(\rho)$ such that $V(\sigma(c))$ contains the majority side of a separation from $\mathcal{T}_M$.
Similarly, let $A \subseteq V(G)$, $|A| \leq r-1$, let $\Sigma'$ be a surface and $\rho'$ be a $\Sigma'$-rendition of $G-A$.
Then we say that $\rho'$ is \emph{$(M-A)$-central} for $G$ if no cell of $\rho'$ contains the majority side of a separation from $\mathcal{T}_M \cap \mathcal{S}_{r-|A|}$.

\paragraph{Weak layouts.}
Given a graph $G$ with a $\Sigma$-rendition $\rho$, we define the \emph{$\rho$-torso}\footnote{We choose this name to distinguish these objects from the more common notion of torsos, which we will define later.} of $G$ to be the graph $T$ with vertex set $N(\rho)$ and the edges resulting from adding an edge $uv$ to $T$ if there exists a non-vortex cell $c \in C(\rho)$ such that $u,v \in N_\rho(c)$.

Let $w$, $k$, $d$, $b$, and $a$ be non-negative integers and let $\Sigma$ be a surface.
We say that a graph $G$ containing a mesh $M$ has a \emph{weak $k$-$(a,b,d)$-$\Sigma$-layout centred at $M$} if there exists a set $A \subseteq V(G)$ with $|A| \leq a$, a submesh $M' \subseteq M$, and a $\Sigma$-rendition of $G - A$, such that
    \begin{enumerate}
        \item $\rho$ is \emph{$(M-A)$-central},
        \item $\rho$ has breadth at most $b$ and depth at most $d$,
        \item there exists a $w$-submesh $M' \subseteq M$, with $w \geq a + b(2d + 1) + 7$,
        \item $M'$ is flat in $\rho$,
        \item each vortex of $\rho$ has a linear decomposition of adhesion at most $d$, and
        \item the $\rho$-torso of $G - A$ has a 2-cell-embedding induced by $\rho$ in $\Sigma$ with representativity at least $k$.
    \end{enumerate}
We call $\rho$ the \emph{layout-rendition}.

This now finally allows us to state a version of the Local Structure Theorem.

\begin{proposition}[Gorsky, Seweryn, and Wiederrecht \cite{GorskySW2025Polynomial}]\label{prop_lst}
    There exist functions $\mathsf{apex}_{\ref{prop_lst}},\mathsf{depth}_{\ref{prop_lst}},\mathsf{mesh}_{\ref{prop_lst}}\colon\mathbb{N}^2\to\mathbb{N}$ such that for all integers $k \geq 1$ and $t \geq 5$, every graph $H$ on $t$ vertices, every graph $G$ and every $\mathsf{mesh}_{\ref{prop_lst}}(t,k)$-mesh $M \subseteq G$ one of the following holds.
    \begin{enumerate}
    \item $G$ has an $H$-minor, or

    \item $G$ and $M$ have a weak $(\mathsf{apex}_{\ref{prop_lst}}(t,k),\nicefrac{1}{2}(t-3)(t-4),\mathsf{depth}_{\ref{prop_lst}}(t,k))$-$\Sigma$-layout centred at $M$ in a surface $\Sigma$ into which $H$ does not embed.\footnote{The fact that last point in the definition of weak layouts holds is a consequence of the explicit construction of the $\Sigma$-rendition in the proof of Theorem 15.1 in \cite{GorskySW2025Polynomial}. In particular, the representativity of the embedding is guaranteed by the existence of the extended large surface wall in the embedded part of the graph.}
    \end{enumerate}
Moreover, it holds that

{\centering
  $ \displaystyle
    \begin{aligned}
        \mathsf{apex}_{\ref{prop_lst}}(t,k),~ \mathsf{depth}_{\ref{prop_lst}}(t,k) \in \mathcal{O}\big((t,k)^{112}\big), \text{ and } \mathsf{mesh}_{\ref{prop_lst}}(t,k) \in \Ocal\big((t+k)^{115} \big) .
    \end{aligned}
  $
\par}

In particular, we have $\mathsf{mesh}_{\ref{prop_lst}}(t,k) \geq 2\mathsf{depth}_{\ref{prop_lst}}(t,k) + \mathsf{apex}_{\ref{prop_lst}}(t,k) + 1$.

There also exists an algorithm that, given $t$, $k$, a graph $H$, a graph $G$ and a mesh $M$ as above as input finds one of these outcomes in time $\poly(t+k)\cdot |E(G)|\cdot |V(G)|^2$.
\end{proposition}

We note that the $\rho$-torsos we are constructing here are not necessarily minors of the graph we started with.
This is a problem we will now work on fixing by reproving the GMST.

\subsection{A variant of the Graph Minor Structure Theorem}\label{sec_GMSTvariant}
We now present a variant of the GMST, which differs from the core result in \cite{GorskySW2025Polynomial} only in the fact that the $\rho$-torsos of the bags of the tree-decomposition without their vortices are minors of the whole graph.
Here we are using the more common notion of torso from structure theory defined as follows:
Given a set $X \subseteq V(G)$, the \emph{torso} of $X$ in $G$ is the graph obtained from $G[X]$ by adding an edge $uv$ to $X$ if $u,v \in X$ and there exists a component $C$ in $G - X$ such that $u,v \in N_G(V(C))$.

For integers $w$, $b$, and $a$, we say that a graph $G$ has a \emph{strong $(a,b,w)$-near embedding} in a surface $\Sigma$ if there exist a set $A \subseteq V(G)$ called the \emph{apex set}, with $|A| \leq a$, and a $\Sigma$-rendition $\rho$ for $G - A$ of breadth at most $b$ where
\begin{enumerate}
    \item all vertices of $G - A$ that do not belong to the interior of a vortex of $\rho$ are grounded,\footnote{It is possible that $V(G)\subseteq A$, where $A$ is the apex set as above. In this case we assume $\Sigma$ to be the \emph{empty surface}. Only the empty graph embeds in the empty surface.}
    \item every vortex of $\rho$ has a linear decomposition of width at most $w$, and
    \item if $G'$ is the $\rho$-torso of $G-A$, then the result of removing $V(\sigma(c)) \setminus N(c)$ for each vortex $c \in C(\rho)$ from $G'$ is a minor of $G-A$ and each of its components has a 2-cell-embedding in a surface with genus at most that of $\Sigma$.
\end{enumerate}

\begin{theorem}\label{thm_GMST}
There exist functions $\mathsf{adhesion}_{\ref{thm_GMST}}, \mathsf{apex}_{\ref{thm_GMST}}, \mathsf{width}_{\ref{thm_GMST}} \colon \mathbb{N} \to \mathbb{N}$ such that for every graph $H$ on $t \geq 1$ vertices and every graph $G$ one of the following holds:
\begin{enumerate}
    \item $G$ contains $H$ as a minor, or
    \item there exists a tree-decomposition $(T,\beta)$ for $G$ with $|V(T)| \in \Ocal(|V(G)|)$ and of adhesion at most $\mathsf{adhesion}_{\ref{thm_GMST}}(t)$ such that for every $x \in V(T)$ the torso $G_x$ of $\beta(x)$ has a strong $(\mathsf{apex}_{\ref{thm_GMST}}(t),\nicefrac{1}{2}(t-3)(t-4),\mathsf{width}_{\ref{thm_GMST}}(t))$-near embedding, with the apex set $A$, into a surface where $H$ does not embed.
\end{enumerate}
Moreover, it holds that

{\centering
  $ \displaystyle
    \begin{aligned}
        \mathsf{adhesion}_{\ref{thm_GMST}}(t),\mathsf{apex}_{\ref{thm_GMST}}(t),~ \mathsf{width}_{\ref{thm_GMST}}(t) \in \Ocal\big(t^{2300}\big).
    \end{aligned}
  $
\par}
There also exists an algorithm that, given $H$ and $G$ as input, finds either an $H$-minor model in $G$ or a tree-decomposition $(T,\beta)$ as above in time $2^{\poly(t)}|V(G)|^{3}|E(G)|\log |V(G)|$.
\end{theorem}

Towards a proof of this theorem, we will need several more notions and useful tools.

\paragraph{Highly linked sets.}
Let $\alpha \in [2/3, 1)_{\mathbb{R}}$.
Moreover, let $G$ be a graph and $X \subseteq V(G)$ be a vertex set. 
A set $S \subseteq V(G)$ is said to be an \emph{$\alpha$-balanced separator} for $X$ if for every component $C$ of $G - S$ it holds that $|V(C) \cap X| \leq \alpha|X|$. 
Let $k$ be a non-negative integer.
We say that $X$ is a \emph{$(k, \alpha)$-linked set} of $G$ if there is no $\alpha$-balanced separator of size at most $k$ for $X$ in $G$.
Given a $(3k, \alpha)$-linked set $X$ of $G$ we define $$\mathcal{T}_{X} \coloneqq \{ (A, B) \in \mathcal{S}_{k+1}(G) ~\!\colon\!~ |X \cap B| > \alpha|X| \}.$$ 
It is not hard to see that $\mathcal{T}_{S}$ is a tangle of order $k+1$ in $G$.

Highly linked sets give us an algorithmic way to find large meshes.
In particular, given a highly linked set $X$, we can find a large mesh whose tangle is a truncation of the tangle induced by $X$.

\begin{proposition}[Thilikos and Wiederrecht \cite{ThilikosW2024Excluding} (see Theorem 4.2.)]\label{prop_algogrid}
Let $k \geq 3$ be an integer and $\alpha \in [2/3,1)$.
There exist universal constants $c_1, c_2 \in \mathbb{N} \setminus \{ 0 \}$, and an algorithm that, given a graph $G$ and a $(c_1k^{20}, \alpha)$-linked set $X \subseteq V(G)$ computes in time $2^{\Ocal(k^{c_2})}|V(G)|^2|E(G)|\log(|V(G)|)$ a $k$-mesh $W \subseteq G$ such that $\mathcal{T}_W$ is a truncation of $\mathcal{T}_X$.
\end{proposition}

We will also need to be able to find balanced separators efficiently.
For this purpose we make use of a tool due to Reed.

\begin{proposition}[Reed \cite{Reed1992Finding}]\label{prop_findsep}
    There exists an algorithm that takes as input an integer $k$, a graph $G$, and a set $X \subseteq V(G)$ of size at most $3k+1$, and finds, in time $2^{\Ocal(k)}m$, either a $\nicefrac{2}{3}$-balanced separator of size at most $k$ for $X$ or correctly determines that $X$ is $(k,\nicefrac{2}{3})$-linked in $G$.
\end{proposition}

\paragraph{Quasi-4-connected components.}
As mentioned earlier, we will want to ensure that the torsos we construct in \zcref{prop_lst} are actually minors of the original graph.
Luckily, there is a tool due to Grohe \cite{Grohe2016Quasi4Connected} that allow us to do this with relative ease by restricting ourselves to ``quasi-4-connected components'' of our graph (see also \cite{CarmesinK2023Characterising}).

We call a graph $G$ \emph{quasi-4-connected} if it is 3-connected and for all separations $(A,B)$ in $G$ of order 3 the set $A \setminus B$ or $B \setminus A$ only contains a single vertex.

\begin{proposition}[Grohe \cite{Grohe2016Quasi4Connected}]\label{prop_quasi4con}
    Every graph has a tree-decomposition $(T,\beta)$ of adhesion at most 3 such that for all $t \in V(T)$ the torso of $\beta(t)$ in $G$ is a minor of $G$ that is either quasi-4-connected or isomorphic to a complete graph of order at most 4.
    
    Furthermore, $(T,\beta)$ can be computed in $\Ocal(n^3)$-time.
\end{proposition}

\paragraph{A proof of our variant of the GMST.}
Now we can prove \zcref{thm_GMST} via a stronger version of it that we in turn prove by induction.
The structure of this proof is derived from one of the core proofs of \cite{RobertsonS1991Graph}.

\begin{theorem}\label{thm_GMST_induction}
    There exist functions $\mathsf{link}_{\ref{thm_GMST_induction}},\mathsf{apex}_{\ref{thm_GMST_induction}},\mathsf{width}_{\ref{thm_GMST_induction}}\colon\mathbb{N}\to\mathbb{N}$ such that for every graph $H$ on $t \geq 1$ vertices, every graph $G$, and every vertex set $X\subseteq V(G)$ with $|X| \leq 3\mathsf{link}_{\ref{thm_GMST_induction}}(t)+1$ one of the following holds:
\begin{enumerate}
    \item $G$ contains $H$ as a minor, or
    \item there exists a rooted tree-decomposition $(T,r,\beta)$ for $G$ with $|V(T)| \in \Ocal(|V(G)|)$ and of adhesion at most 
    \begin{align*}
    3\mathsf{link}_{\ref{thm_GMST_induction}}(t)+\mathsf{apex}_{\ref{thm_GMST_induction}}(t)+\mathsf{width}_{\ref{thm_GMST_induction}}(t)+3
    \end{align*}
    such that 
    \begin{enumerate}
        \item for every $x \in V(T)$ the torso $G_x$ of $\beta(x)$ has a strong $(4\mathsf{link}_{\ref{thm_GMST_induction}}(t)+\mathsf{apex}_{\ref{thm_GMST_induction}}(t),\nicefrac{1}{2}(t-3)(t-4),2\mathsf{link}_{\ref{thm_GMST_induction}}(t)+\mathsf{width}_{\ref{thm_GMST_induction}}(t))$-near embedding, with the apex set $A$, into a surface where $H$ does not embed, and
        \item let $A_r$ be the apex set for the torso $G_r$ at $r$, then $X \subseteq A_r$.
    \end{enumerate}
\end{enumerate}
Moreover, it holds that

{\centering
  $ \displaystyle
    \begin{aligned}
        \mathsf{apex}_{\ref{thm_GMST_induction}}(t),~ \mathsf{width}_{\ref{thm_GMST_induction}}(t) \in \Ocal\big(t^{115}\big)\text{ and }\mathsf{link}_{\ref{thm_GMST_induction}}(t)\in\Ocal\big(t^{2300}\big).
    \end{aligned}
  $
\par}
There also exists an algorithm that, given $H$ and $G$ as input, finds either an $H$-minor model in $G$ or a tree-decomposition $(T,\beta)$ as above in time \(2^{\poly(t)}n^3 m \log m\).
\end{theorem}
\begin{proof}
    Our functions for this proof are:
    \begin{align*}
        \mathsf{apex}_{\ref{thm_GMST_induction}}(t)    \coloneqq & \ \mathsf{apex}_{\ref{prop_lst}}(t,4) \\
        \mathsf{width}_{\ref{thm_GMST_induction}}(t)  \coloneqq & \ 2\mathsf{depth}_{\ref{prop_lst}}(t,4) + 1 \\
        \mathsf{link}_{\ref{thm_GMST_induction}}(t)    \coloneqq & \ \mathsf{c}_1 \mathsf{mesh}_{\ref{prop_lst}}(t,4)^{20},
    \end{align*}
    where $\mathsf{c}_1$ is the first constant from \zcref{prop_algogrid}.
    The bound on the order of $\mathsf{link}_{\ref{thm_GMST_induction}}(t)$ is justified by the fact that $\mathsf{mesh}_{\ref{prop_lst}}(t,4) \in \Ocal(t^{115})$.

    We proceed by induction on $|V(G) \setminus X|$ and observe that in case $|V(G)| \leq 3\mathsf{link}_{\ref{thm_GMST_induction}}(t) + 1$, we may simply take $T$ to be the tree on one vertex $r$ and set $\beta(t) \coloneqq V(G)$.
    This satisfies the second part of our statement if we simply put the entirety of the graph into the apex set for the root.
    Thus we may move on assuming $|V(G)| \geq 3\mathsf{link}_{\ref{thm_GMST_induction}}(t) + 2$.

    Suppose next that $|X| \leq 3\mathsf{link}_{\ref{thm_GMST_induction}}(t)$.
    This allows us to take some vertex $v \in V(G) \setminus X$ and let $X' \coloneq X \cup \{ v \}$.
    We may now apply the induction hypothesis to $G$ and $X'$, since $|V(G) \setminus X'| < |V(G) \setminus X|$, which results in us either finding $H$ as a minor, or the rooted tree-decomposition we desire.
    Hence, we may assume that $|X| = 3\mathsf{link}_{\ref{thm_GMST_induction}}(t) + 1$.

    Applying \zcref{prop_findsep} to $X$ in $G$ now yields one of two possible outcomes in $(2^{\mathsf{poly}(t)}m)$-time:
    \begin{enumerate}
        \item A $\nicefrac{2}{3}$-balanced separator $S$ for $X$ in $G$ with $|S| \leq \mathsf{link}_{\ref{thm_GMST_induction}}(t)$, or
        \item $X$ is $(\mathsf{link}_{\ref{thm_GMST_induction}}(t), \nicefrac{2}{3})$-linked in $G$.
    \end{enumerate}

    \textbf{Case 1:}
    There exists a $\nicefrac{2}{3}$-balanced separator $S$ for $X$ in $G$ with $|S| \leq \mathsf{link}_{\ref{thm_GMST_induction}}(t)$.

    Let $G_1', \ldots , G_\ell'$ be the components of $G - S$ and for each $i \in [\ell]$, let $G_i \coloneqq G[V(G_i') \cup S]$.
    Moreover, for each $i \in [\ell]$, let $X_i' \coloneqq (V(G_i) \cap X) \cup S$.
    By definition of balanced separators, we have
    \begin{align*}
        |X_i'|  \leq & \ \lfloor \nicefrac{2}{3} \cdot  (3\cdot \mathsf{link}_{\ref{thm_GMST_induction}}(t) + 1) \rfloor + \mathsf{link}_{\ref{thm_GMST_induction}}(t) \\
                \leq & \ 3\cdot \mathsf{link}_{\ref{thm_GMST_induction}}(t) .
    \end{align*}
    This allows us to define a rooted tree-decomposition $(T,r,\beta)$ with the desired properties, in the absence of an $H$-minor, as follows.
    First, we introduce the vertex $r$ and set $\beta(r) \coloneqq X \cup S$.
    Due to their relatively small sizes, we may let all of $\beta(r)$ be the apex set for $r$.

    For each $i \in [\ell]$ where $V(G_i) = X_i'$, we may then simply introduce a vertex $x_i$ to the tree, make it adjacent to $r$, and set $\beta(x_i) \coloneqq X_i'$.
    For the remaining $i \in [\ell]$, there exists some $v_i \in V(G_i) \setminus X_i'$, allowing us to set $X_i = X_i' \cup \{ v_i \}$ whilst ensuring that $|V(G_i) \setminus X_i| < |V(G) \setminus X|$.
    This allows us to apply the induction hypothesis to $G_i$ and $X_i$.
    If this yields $H$ as a minor, we are done.
    Otherwise, there exists a rooted tree-decomposition $(T_i,r_i,\beta_i)$ with the desired properties.

    We may assume that we do not find an $H$-minor this way.
    For each $i \in [\ell]$ for which $v_i$ is defined, we may then add $T_i$ to our tree, make $r_i$ adjacent to $r$, and set $\beta(x) \coloneq \beta_i(x)$ for all $x \in V(T_i)$.
    This concludes the proof of our theorem in this case.

    \textbf{Case 2:}
    $X$ is $(\mathsf{link}_{\ref{thm_GMST_induction}}(t), \nicefrac{2}{3})$-linked in $G$.

    We start by applying \zcref{prop_algogrid} to find a $\mathsf{mesh}_{\ref{prop_lst}}(t,4)$-mesh $M$ in $G$ such that the tangle $\mathcal{T}_M$ is a truncation of the tangle $\mathcal{T}_X$, which takes $(2^{\mathsf{poly}(t)}n^2m \log m)$-time.
    This allows us to continue by applying \zcref{prop_lst} to $M$, which takes $\mathsf{poly}(t)n^2m$-time.
    If this results in an $H$-minor model, we are done.
    
    Hence, we must instead find a weak 4-$(\mathsf{apex}_{\ref{prop_lst}}(t,4),\nicefrac{1}{2}(t-3)(t-4),\mathsf{depth}_{\ref{prop_lst}}(t,4))$-$\Sigma$-layout for $G$ and $M$ centred at $M$ in a surface $\Sigma$ into which $H$ does not embed, with the layout-rendition $\rho$.
    According to the definition of weak layouts, there therefore exists a $w$-mesh $M' \subseteq M - A$, with $w \geq \mathsf{apex}_{\ref{prop_lst}}(t) + \nicefrac{1}{2}(t-3)(t-4)\mathsf{depth}_{\ref{prop_lst}}(t) + 7$, which defines a tangle $\mathcal{T}_{M'}$ of order $w$.

    We apply \zcref{prop_quasi4con} to find a tree-decomposition $(T',\beta')$ of $G - A$ with adhesion at most 3 such that for all $t \in V(T')$ the torso of $\beta(t)$ in $G$ is a minor of $G$ that is either quasi-4-connected or isomorphic to $K_4$.
    This again only takes $\Ocal(n^3)$-time.
    Since $\mathcal{T}_{M'}$ orients all separations of order less than $w$, it in particular orients all of the separations induced by the adhesion sets of $(T',\beta')$.
    As $T'$ is a tree, there therefore exists a unique $t \in V(T')$ such that $\beta(t)$ is on the big side of all separations of order at most three with respect to $\mathcal{T}_{M'}$.
    We let $G^\star$ be the torso of $\beta(t)$.

    Our goal is to now show that we can derive a weak 1-$(\mathsf{apex}_{\ref{prop_lst}}(t,4),\nicefrac{1}{2}(t-3)(t-4),\mathsf{depth}_{\ref{prop_lst}}(t,4))$-$\Sigma$-layout for $G^\star$ and a $w$-mesh $M^\star \subseteq G^\star$ from $\rho$.
    Let $\psi$ be the 2-cell-embedding of the $\rho$-torso $G'$ of $G-A$, which exists according to the fact that $\rho$ is a weak layout.
    In particular, $\psi$ has representativity 4.
    Now let $(A,B)$ be a separation of order at most three such that $V(G^\star) \subseteq B$ and let $S = A \cap B$.
    Note that if $S \setminus V(G') \neq \emptyset$, there exists some vortex $c \in C(\rho)$ such that $S \subseteq \sigma(c)$ and thus, we do not really care about these types of separations.

    We may therefore suppose that $S \subseteq V(G')$ and note that $S$ corresponds to a closed curve $\gamma$ in $\Sigma$ that intersects at most 3 vertices in $N(\rho)$.
    As $\psi$ has representativity at least 4, this cannot be a non-contractible curve and thus we may modify $\rho$ (and $\psi$) by turning the disk bounded by $\gamma$ into a cell that contains $A$ and has the vertices that $\gamma$ intersects on its boundary.
    Clearly, this new cell is not a vortex and we may refine $\rho$ and $\psi$ using all of these separations.
    The result is a rendition $\rho'$ that has an associated embedding $\psi'$ of the $\rho'$-torso $G''$ of $G-A$ in $\Sigma$, such that $G''$ is a minor of $G-A$.
    Regarding $M^\star$, we note that meshes are so well connected, that it is impossible for us to separate any substantial chunk out of the mesh using any separation with less than 4 vertices.
    Thus, we may simply adjust $M$ whilst we are refining $\rho$ and $\psi$ to derive a $w$-mesh from $M$ as desired.
    By construction $\rho'$ is therefore $(M^\star - A)$-central.

    From here we continue our work with $\rho'$.    
    We let $A_r \coloneqq A \cup X$ and observe that
    \[ |A_r| \leq |A \cup X| \leq 3\mathsf{link}_{\ref{thm_GMST_induction}}(t) + \mathsf{apex}_{\ref{prop_lst}}(t) + 1 . \]
    Let $c_1, \ldots , c_q$, with $q \leq \nicefrac{1}{2}(t-3)(t-4)$, be the vortices of $\rho'$.
    Since $\rho'$ and $\rho$ have the same vortices, there exists a linear decomposition $(Y_1, \ldots , Y_{\ell_i})$ of adhesion at most $\mathsf{depth}_{\ref{prop_lst}}(t,4)$ for each $i \in [q]$.
    Since $\rho'$ is $(M^\star- A)$-central and $\mathcal{T}_M$ is a truncation of $\mathcal{T}_X$, this further implies that, for all $i \in [q]$ and $j \in [\ell_i]$, we have
    \[ |Y_j^i \cap X| \leq \nicefrac{2}{3}|X| \leq 2\mathsf{link}_{\ref{thm_GMST_induction}}(t) . \]
    Let $U$ be the union of $N(\rho')$ and all vertices appearing in a set $Y_j^i \cap Y_{j+1}^i$ for some $i \in [q]$ and $j \in [\ell_i]$.
    We will now construct our desired rooted tree-decomposition $(T,r,\beta)$.
    First, we set $\beta(r) \coloneq U \cup A_r$.
    Using our earlier observations, we deduce that, for all $i \in [q]$ and $j \in [\ell_i]$, we have 
    \begin{align*}
        |Y_j^i \cap \beta(r)|   & \leq 2\mathsf{link}_{\ref{thm_GMST_induction}}(t) + 2\mathsf{depth}_{\ref{prop_lst}}(t,4) + 1 \\
                                & \leq 2\mathsf{link}_{\ref{thm_GMST_induction}}(t) + \mathsf{width}_{\ref{thm_GMST_induction}}(t) .
    \end{align*}
    For each $i \in [q]$ and $j \in [\ell_i]$, we let $F_{i,j} \coloneqq G[A \cup Y^i_j]$.
    Then
    \begin{align*}
        |V(F_j^i) \cap \beta(r)|    & \leq 2\mathsf{link}_{\ref{thm_GMST_induction}}(t) + \mathsf{width}_{\ref{thm_GMST_induction}}(t) + \mathsf{apex}_{\ref{thm_GMST_induction}}(t) \\
                                    & \leq 3\mathsf{link}_{\ref{thm_GMST_induction}}(t) ,
    \end{align*}
    since we have $\mathsf{mesh}_{\ref{prop_lst}}(t,4) \geq \mathsf{width}_{\ref{thm_GMST_induction}}(t) + \mathsf{apex}_{\ref{thm_GMST_induction}}(t)$ (see \zcref{prop_lst}).
    This allows us to set $X_{i,j}' \coloneqq V(F_{i,j} \cap \beta(r)$ for the coming arguments.
    As in Case 1, there are two cases to be considered.
    Should we have $V(F_{i,j}) = X_{i,j}'$, we introduce a vertex $r_{i,j}$ adjacent to $r$ and set $\beta(r_{i,j}) = V(F_{i,j})$.
    Otherwise, there exists some $v_{i,j} \in V(F_{i,j}) \setminus X_{i,j}'$ and we can set $X_{i,j} \coloneqq X_{i,j}' \cup \{ v_{i,j} \}$, which means that $|V(F_{i,j}) \setminus X_{i,j}| < |V(G) \setminus X|$ for all $i \in [q]$ and $j \in [\ell_i]$.
    Therefore, our induction hypothesis either yields an $H$-minor or a rooted tree-decomposition $(T_{i,j}, r_{i,j}, \beta_{i,j})$ that satisfies the second outcome of our statement.
    If we find an $H$-minor, we are done.
    Thus we can add the trees $T_{i,j}$ to $T$, make $r_{i,j}$ adjacent to $r$, and set $\beta(x) \coloneq \beta_{i,j}(x)$ for each $x \in V(T_{i,j})$.
    This finishes our arguments for the vortices of $\rho'$.

    We must now discuss how the graphs $J_c' \coloneqq \sigma_{\rho'}(c)$, for the non-vortex cells $c \in C(\rho')$, are decomposed in our tree-decomposition, if we fail to find an $H$-minor.
    Recall that $\rho'$ is $(M^\star - A)$-central and $\mathcal{T}_M$ is a truncation of $\mathcal{T}_X$.
    Thus $J_c'$ cannot contain more than $2\mathsf{link}_{\ref{thm_GMST_induction}}(t)$ vertices of $X$.
    We may therefore set $J_c \coloneqq G[A \cup V(J_c')]$ and $X_c' \coloneqq V(J_c) \cap A_r$.
    It follows that
    \begin{align*}
        |X_c'|  & \leq 2\mathsf{link}_{\ref{thm_GMST_induction}}(t) + \mathsf{apex}_{\ref{thm_GMST_induction}}(t) \\
                & \leq 3\mathsf{link}_{\ref{thm_GMST_induction}}(t).
    \end{align*}
    For each non-vortex cell $c \in C(\rho')$ where $V(J_c) = X_c'$, we introduce a vertex $r_c$ adjacent to $r$ and set $\beta(r_c) \coloneqq X_c'$.
    This leaves us to deal with those non-vortex cells $c$ for which we can find a vertex $v_c \in V(J_c) \setminus X_c'$.
    Analogous to our previous constructions, we set $X_c \coloneqq X_c' \cup \{ v_c \}$ and thus have $|V(J_c) \setminus X_c| < |V(G) \setminus X|$.
    Our induction hypothesis tells us that we either find an $H$-minor - and are thus done - or we find a rooted tree-decomposition $(T_c, r_c, \beta_c)$ for each $J_c$ satisfying the conditions laid out in the second option of our statement.
    As before, we introduce the trees $T_c$ to $T$, make $r_c$ adjacent to $r$, and set $\beta(x) \coloneqq \beta_c(x)$ for each $x \in V(T_c)$.
    This finally completes our construction of the rooted tree-decomposition $(T,r,\beta)$ with the outcomes we desire.

    Our arguments can be realised as a recursive algorithm, taking a total of $n$ recursive calls to find the tree-decomposition we want.
    Each call takes $(2^{\mathsf{poly}(t)}n^2m \log n)$-time, leading to a total runtime in $2^{\mathsf{poly}(t)}n^3m \log n$.
\end{proof}

\paragraph{Better runtimes for our GMST-variant through randomization.}
Notably, using the same arguments used in Section 17 of \cite{GorskySW2025Polynomial}, we can provide an algorithm running in time $(t + n)^{\Ocal(1)}$ that still guarantees polynomial bounds on our parameters.
The trade-off here is that the bounds on our functions get noticably worse, and the associated algorithm is randomised and returns the outcomes we desire only with high probability instead of certainty.
We refer the reader to \cite{GorskySW2025Polynomial} for an account of how we can modify our proof of \zcref{thm_GMST_induction} to get the desired results, whilst noting that this approach uses results from \cite{ChekuriKS2005Multicommodity,Amir2010Approximation,ChekuriE2013Polylogarithmic,ChekuriC2016Polynomial,KawarabayashiTW2021Quickly,ThilikosW2024Excluding}.

\begin{theorem}\label{thm_GMST_randomised}
There exist functions $\mathsf{adhesion}_{\ref{thm_GMST_randomised}}, \mathsf{apex}_{\ref{thm_GMST_randomised}}, \mathsf{width}_{\ref{thm_GMST_randomised}} \colon \mathbb{N} \to \mathbb{N}$ such that for every graph $H$ on $t \geq 1$ vertices and every graph $G$ one of the following holds:
\begin{enumerate}
    \item $G$ contains $H$ as a minor, or
    \item there exists a tree-decomposition $(T,\beta)$ for $G$ with $|V(T)| \in \Ocal(|V(G)|)$ and of adhesion at most $\mathsf{adhesion}_{\ref{thm_GMST_randomised}}(t,k)$ such that for every $x \in V(T)$ the torso $G_x$ of $\beta(x)$ has a strong $(\mathsf{apex}_{\ref{thm_GMST_randomised}}(t),\nicefrac{1}{2}(t-3)(t-4),\mathsf{width}_{\ref{thm_GMST_randomised}}(t))$-near embedding into a surface, with the apex set $A$, where $H$ does not embed, such that $G_x - A$ is a minor of $G$.
\end{enumerate}
Moreover, it holds that

{\centering
  $ \displaystyle
    \begin{aligned}
        \mathsf{adhesion}_{\ref{thm_GMST_randomised}}(t),\mathsf{apex}_{\ref{thm_GMST_randomised}}(t),~ \mathsf{width}_{\ref{thm_GMST_randomised}}(t) \in \Ocal\big(t^{11500}\big).
    \end{aligned}
  $
\par}
There also exists a randomised algorithm that, given $H$ and $G$ as input, finds either an $H$-minor model in $G$ or a tree-decomposition $(T,\beta)$ as above with high probability in time $(t+n)^{\Ocal(1)}$.
\end{theorem}

\section{Improved bounds on the relationship between clique- and grid-minors}\label{sec_cliquesandgrids}
We now begin by tackling the task of determining the branchwidth of a graph that excludes both the $K_t$ and the $k$-grid as a minor.
For this purpose we first introduce more refined tools concerning the branchwidth of graphs embedded in surfaces and their associated branch-decompositions.
Afterwards we switch immediately into the proof of the main result of this section, which is~\zcref{thm_woodsquestion}.

\subsection{Sphere-cut decompositions and more tools}
We will need a few tools from the surrounding theory, particularly concerning (partially) embedded graphs on surfaces.
Our principal goal will be to reprove the following result from \cite{ThilikosW2025Approximating} in a way that allows us to extract a branch-decomposition.

\begin{proposition}[Thilikos and Wiederrecht \cite{ThilikosW2025Approximating}]\label{prop_spherewithoutvorticeshighbw}
    Let $G$ be a 2-connected graph with $\mathbf{bw}(G) \geq 2$ and a $\Sigma$-rendition $\rho$, where $\Sigma$ is the sphere, such that $\rho$ has breadth $b$, every vortex of $\rho$ has a linear decomposition of width at most $w$, and for every non-vortex cell $c \in N(\rho)$ we have $|N(c)| = 2$ and $V(\sigma(c)) - N(c) = \emptyset$.
    Further let $G'$ be the result of deleting $V(\sigma(c')) \setminus N(c')$ from $G$ for each vortex $c' \in C(\rho)$.
    Then $\bw(G) \leq \bw(G') + 2wb$.
\end{proposition}

For this purpose, we need to discuss \zcref{prop_catchrat} in more depth, as the ideas used to prove this result actually provide a branch-decomposition with some particularly strong properties after some more massaging.

Let $G$ be a graph with an embedding $\phi$ in the sphere $\Sigma$ and let $F$ be the set of its faces.
A curve $\gamma$ in $\Sigma$ is called a \emph{$\phi$-noose}, or simply a \emph{noose} if $\phi$ is clear from the context, if $\gamma$ intersects the drawing $\phi$ only in the vertices of $G$ and furthermore, for each face $f \in F$, the intersection of $\gamma$ and $f$ is either empty or contains at most one component.
The tuple $(T,\tau)$ is called a \emph{sphere-cut decomposition} of $G$ in $\Sigma$, if $(T,\tau)$ is a branch-decomposition of $G$ and additionally, for every edge $e \in E(T)$ inducing the partition $X_e,Y_e$ of $E(G)$, there exists a noose $\gamma_e$ such that the set of vertices $V(\gamma_e)$ that $\gamma_e$ intersects is equal to $\partial_G(X_e)$.

By combining the core results of \cite{SeymourT1994Call} and \cite{GuT2008Optimal}, the following can be derived.

\begin{proposition}[Dorn, Penninkx, Bodlaender, and Fomin \cite{DornPBF10Efficient}]\label{prop:findscdecomp}
    Let $G$ be a graph without vertices of degree one that has an embedding $\phi$ in the sphere.
    There exists an algorithm that finds a sphere-cut decomposition of $G$ with width at most $\bw(G)$ in time $\mathcal{O}(n^3)$.
\end{proposition}

We are now ready to reprove \zcref{prop_spherewithoutvorticeshighbw} in a constructive way.
This result will later be used to compute branch-decompositions with appropriate width for $H$-minor-free graphs.
The bound we give on the width of the branch-decomposition we find is mildly worse than what \zcref{prop_spherewithoutvorticeshighbw} offers us.
Nonetheless, for our purposes, the additional factor will completely disappear in our estimates on all functions we base on this theorem.
Incidentally this is also the only proof in which we need to deal with sphere-cut decompositions.

\begin{theorem}\label{thm_spherewithoutvorticeshighbw}
    Let $G$ be a 2-connected graph with $\mathbf{bw}(G) \geq 2$ and a $\Sigma$-rendition $\rho$, where $\Sigma$ is the sphere, such that $\rho$ has breadth $b$, every vortex of $\rho$ has a linear decomposition of width at most $w$,\footnote{We assume that we are provided with these linear decompositions.} and for every non-vortex cell $c \in N(\rho)$ we have $|N(c)| = 2$ and $V(\sigma(c)) - N(c) = \emptyset$.
    Further let $G'$ be the result of deleting $V(\sigma(c')) \setminus N(c')$ from $G$ for each vortex $c \in C(\rho)$.
    Then there exists an algorithm that finds a branch-decomposition of $G$ with width at most $\bw(G') + 2wb + 6b$ in $\Ocal(bwm + n^3)$-time.
\end{theorem}
\begin{proof}
    We may assume that $\rho$ has at least one vortex, as the statement otherwise trivially holds.
    Thus $b$ is positive.
    Furthermore, by definition of a linear decomposition, we also know that $w$ is positive, since each bag of each linear decomposition contains at least one vertex.
    
    First, we want to ensure that every vertex in $N_\rho(c)$ for any vortex $c \in C(\rho)$ is incident to an edge drawn outside of the vortices.
    Let $v_1, v_2, \ldots , v_\ell$ be a linearization of the order $\Omega_c$ originating from the vortex-society $(G_c, \Omega_c)$ of $c$ in $\rho$.
    We next introduce the edges $v_iv_{i+1}$ and $v_\ell v_1$, unless they are already found in $G'$.
    The graph that results from carrying out this process for all vortices is called $G''$.
    For each vortex $c \in C(\rho)$ we can modify $\rho$ by pushing the drawing of each edge between two vertices in $N_\rho(c)$ into $c$.
    A linear decomposition $(X_1,\ldots, X_\ell)$ of $(G_c,\Omega_c)$ (restricted to $G''$) of width 3 is then simply defined by letting $X_i = \{ v_1, v_i,v_{i+1} \}$ for all $i \in [\ell - 1]$ and $X_\ell = \{ v_1, v_\ell \}$.
    The resulting $\Sigma$-rendition $\rho''$ verifies that $G''$ has a $\Sigma$-rendition in the sphere with breadth $b$ and width 3, such that for every non-vortex cell $c \in N(\rho'')$ we have $|N(c)| = 2$ and $V(\sigma_{\rho''}(c)) - N(c) = \emptyset$.
    According to \zcref{prop_spherewithoutvorticeshighbw}, we have $\bw(G'') \leq \bw(G') + 6b$.

    Next, we observe that by construction of $G''$, there also exists a $\Sigma$-rendition $\rho^\star$ of $G''$ based on $\rho$ such that for every cell $c \in N(\rho^\star)$ we have $|N(c)| = 2$, all vortices of $\rho$ are disjoint from $c$, and $V(\sigma_{\rho^\star}(c)) - N(c) = \emptyset$.
    From this we can derive an embedding $\phi$ of $G''$ in $\Sigma$ such that $\phi$ only intersects the vortices of $\rho$ in the vertices that lie on their boundaries according to $\rho$.
    Our goal will be to now find a branch-decomposition for $G^\star = G'' \cup G$ of width at most $\bw(G') + 2wb + 6b$, which can be used to easily derive a branch-decomposition of the same width for $G$, as $G \subseteq G^\star$ by definition.

    Let $c \in C(\rho)$ be a fixed vortex.
    Since each vortex is bounded by a cycle in $G''$, there exists a block $B$ of $G''$ that contains $N_{\rho}(c)$.
    Using $\phi$ and the fact that $\bw(G'') \leq \bw(G') + 6b$, we use \zcref{prop:findscdecomp} to find a sphere-cut decomposition $(T'',\tau'')$ of $B$.
    We then modify $(T'',\tau'')$ as follows.
    For each vortex $c \in C(\rho)$, let $v_1, \ldots , v_\ell$ be a linearization of $\Omega_c$ from the vortex society $(G_c,\Omega_c)$ in $\rho$ associated with the linear decomposition represented by the sets $(Z_1, \ldots , Z_\ell)$.
    Recall that this linear decomposition has width at most $w$ and thus $|Z_i| \leq w$ for all $i \in [\ell]$.
    Let $E_c \coloneq E(G_c) \setminus E(G'')$ and note that $\sigma_\rho(c) = G_c$ and $E_c \cap E(G'')$ contains only edges that we added on the boundary of the vortices when constructing $G''$.
    We partition $E_c$ into $\ell$ sets $E_1, \ldots , E_\ell$ by letting $E_1 \coloneqq E(G_c[Z_1])$ and for all $i \in [2,\ell]$, letting $E_i \coloneqq E(G_c[Z_i]) \setminus \bigcup_{j=1}^{i-1} E_j$.
    Note that this is a partitioning of $E_c$ according to the properties of a linear decomposition.

    For each $i \in [\ell]$, let $e_i \in E(B)$ be $v_iv_{i+1} \in E(B)$, if $i \in [\ell-1]$, and $v_\ell v_1 \in E(B)$, if $i = \ell$.
    Further, let $u_i \coloneqq \tau''(e_i) \in L(T'')$.
    We construct $T'$ from $T''$ by subdividing the unique edge incident to $u_i$ in $T''$ and naming the subdivision vertex $w_i$, for each $i \in [\ell]$ such that $E_i \neq \emptyset$.
    Next, let $T_i$ be an arbitrary ternary tree with $|E_i|$ leaves, with an arbitrary edge subdivided and with $w_i'$ being the subdivision vertex in $T_i$.
    (If $|E_i| = 0$, we let $T_i$ be the empty graph and if $|E_i| = 1$ we let $w_i'$ be the unique vertex of $T_i$.)
    From this we construct $T$ in the obvious way by taking $T' \cup \bigcup_{i=1}^\ell T_i$ and adding the edge $w_iw_i'$ for each $i \in [\ell]$ with $E_i \neq \emptyset$.
    In case $E_i \neq \emptyset$, we further let $\tau_i$ be an arbitrary bijection from $E_i$ to $L(T_i)$.
    For each $e \in E(B) \cup E_c$, we then set $\tau(e)$ to $\tau''(e)$, if $e \in E(B)$, and to $\tau_i(e)$, if $e \in E_i$.

    The tuple $(T,\tau)$ now satisfies the definition of a branch-decomposition for the graph $B^\star \coloneqq B \cup G_c$.
    Now, we are interested in the width of this new decomposition.
    Every edge $e \in E(T)$ defines a partition $X_e,Y_e$ of the edges in $E(B) \cup E_c$.
    First, note that for any edge $e \in E(T)$ corresponding to an edge in $\bigcup_{i=1}^\ell E(T_i)$, we clearly have $|\partial_{B^\star}(X_e)| \leq w$.
    Furthermore, for any of the edges $e \in E(T)$ incident to $u_i$ for some $i \in [\ell]$, we have $|\partial_{B^\star}(X_e)| \leq 2$, which is acceptable.

    Thus we may now concentrate on an arbitrary edge $e \in E(T)$ such that $e$ is neither incident to $u_i$ for some $i \in [\ell]$, nor contained in $\bigcup_{i=1}^\ell E(T_i)$.
    Note that this in particular implies that $e$ induces a partition $X_e'',Y_e''$ in $(T'',\tau'')$ of the edges in $E(B)$, such that $X_e'' \subseteq X_e$ and $Y_e'' \subseteq Y_e$.\footnote{In case $e = w_iw_i'$ for some $i \in [\ell]$ one of the two sets $X_e'',Y_e''$ is empty.}
    Clearly, by construction of $(T,\tau)$, we have $|\partial_{B^\star}(X_e)| \geq |\partial_{B}(X_e'')|$, since $\partial_{B}(X_e'') \subseteq \partial_{B^\star}(X_e)$.
    We may therefore assume that there exists a $v \in \partial_{B^\star}(X_e) \setminus \partial_{B}(X_e'')$, as we would otherwise be done.
    Let $f \in X_e$ and $g \in Y_e$ be two edges that are both incident to $v$.
    
    Suppose that $f \in E(B)$ and therefore $f \in X_e''$.
    Then, since $v \in \partial_{B^\star}(X_e) \setminus \partial_{B}(X_e'')$, we must have $g \in E_c$ and in particular $v \in N_\rho(c)$.
    This implies that $v \in V(\Omega_c)$.
    We may assume w.l.o.g.\ that $v = v_1$ from our linearization of $\Omega_c$ and furthermore, that $g \in E_1$.
    Since $e$ is not incident to any $u_i$, it is in particular not incident to $u_1$ and thus, other than just having $g \in Y_e$, we also have $e_1 \in Y_e''$ and thus $v \in \partial_{B}(X_e'')$, yielding a contradiction.
    We conclude that $f \in X_e \setminus X_e''$ and $g \in Y_e \setminus Y_e''$, independent of the particular choice of $v \in \partial_{B^\star}(X_e) \setminus \partial_{B}(X_e'')$.

    As $(T'',\tau'')$ is a sphere-cut decomposition and $e$ corresponds to an edge of $T''$, there exists a noose $\gamma$, with $V(\gamma)$ being the vertices of $B$ it intersects, that intersects $c - N_\rho(c)$ in at most one component.
    Let $\Delta_1, \Delta_2$ be the closure of the two components of $\Sigma - \gamma$ and note that $\gamma$ forms the boundary of both of these disks.
    Thus, using the embedding $\phi$ of $B$ in $\Sigma$, there exist two subgraphs $G_1$ and $G_2$ of $B$ which we label such that $G_1 \cup G_2 = B$, $V(G_1 \cap G_2) = V(\gamma)$, $E(G_1) = X_e''$, and $E(G_2) = Y_e''$.
    Recall that $e_i = v_iv_{i+1}$ if $i \in [\ell-1]$ and $e_\ell = v_\ell v_1$, if $i = \ell$.
    
    Due to the fact that $\gamma$ intersects $c - N_\rho(c)$ in at most one component, there either exists a $j \in [2]$ such that $e_i \in E(G_j)$ for all $i \in [\ell]$, or there exist two distinct $j,h \in [\ell]$, with $j < h$, such that $e_1, \ldots , e_j \in E(G_1)$, $e_{j+1}, \ldots , e_h \in E(G_2)$, and $e_{h+1}, \ldots , e_\ell \in E(G_1)$.
    By construction of $(T,\tau)$, the first option implies that we have $E_c \subseteq X_e''$ or $E_c \subseteq Y_e''$, which means that $\partial_{B^\star}(X_e) \setminus \partial_{B}(X_e'') = \emptyset$ and thus we have confirmed that $|\partial_{B^\star}(X_e)| \leq \bw(G') + 6b$.
    
    Otherwise, since we know that $e$ is an edge of $T$ that is not incident to $u_i$ in $T$ and still corresponds to an edge of $T''$, we know that $e_i \in X_e$ if and only if $E_i \subseteq X_e$, again by construction of $(T,\tau)$.
    Thus, we have $\partial_{B^\star}(X_e) \setminus \partial_{B}(X_e'') \subseteq (Z_j \cap Z_{j+1}) \cup (Z_h \cap Z_{h+1})$ and, since the width of $\rho$ is $w$, we have $(Z_j \cap Z_{j+1}) \cup (Z_h \cap Z_{h+1}) \leq 2w$.
    Hence, we have $|\partial_{B^\star}(X_e)| \leq \bw(G') + 2w + 6b$.
    Thus $(T,\tau)$ has width at most $\bw(G') + 2w + 6b$.

    We can now repeat this construction for all vortices whose boundaries are contained in $V(B)$, yielding a branch-decomposition of $B$ together with its associated vortices of width at most $\bw(G') + 2wb' + 6b$, where $b' \leq b$ is the number of vortices belonging to $B$.
    Using arguments analogous to the proof of \zcref{lem_branchwidthofblocks}, we observe that we can repeat this procedure for every block of $G^\star$ and then easily combine the resulting branch-decompositions into a branch-decomposition of $G^\star$ of width at most $\bw(G') + 2wb + 6b$ in time $\mathcal{O}(bwm)$.
\end{proof}

To interact with \zcref{thm_GMST}, we also need an observation about the location of parts of the graph with high branchwidth, if the graph has a tree-decomposition with low adhesion.

\begin{proposition}[Thilikos and Wiederrecht \cite{ThilikosW2025Approximating}]\label{prop_findhighbwbag}
    Let $c$ be a positive integer, let $G$ be a graph with $\bw(G) > c$, and let $(T,\beta)$ be a tree-decomposition of $G$ with adhesion at most $c$.
    Then the branchwidth of $G$ is at most the maximum of the branchwidth of the torsos of the bags of $(T,\beta)$.
\end{proposition}

\subsection{Finding cliques or grids}
We first establish the following function:
\begin{align*}
    \mathbf{b}_{\ref{thm_woodsquestion}}(t,k,g)   \coloneqq~& 4k + \max(g,0) (4k + (t-3)(t-4)\mathsf{width}_{\ref{thm_GMST}}(t)) \ + \\
     & \max(\mathsf{adhesion}_{\ref{thm_GMST}}(t), \mathsf{apex}_{\ref{thm_GMST}}(t)) + \nicefrac{1}{2}\mathsf{width}_{\ref{thm_GMST}}(t)(t-3)(t-4) .
\end{align*}
Additionally, recall the following classic theorem on the Euler genus of $K_t$.
\begin{proposition}[Ringel and Young \cite{RingelY1968Solution,Ringel1974Map}]\label{prop_ktgenus}
    If $t \geq 3$ is an integer, the Euler genus of $K_t$ is
    \[ \left\lceil \frac{(t-3)(t-4)}{6} \right\rceil . \]
\end{proposition}

Using \zcref{thm_GMST} and \zcref{prop_ktgenus}, we can conclude that $\mathbf{b}_{\ref{thm_woodsquestion}}(t,k, \mathsf{eg}(K_t)) \in \Ocal(t^2k + t^{2304})$.

Moving on to the proof of our main theorem, we note that we phrase this theorem in terms of $H$-minors instead of fixing $H$ to be $K_t$ to get a slightly more general result.

\begin{theorem}\label{thm_woodsquestion}
    Let $k$ be a positive integer, let $H$ be a graph with $t \coloneqq |V(H)|$, and let $G$ be a graph $\bw(G) \geq \mathbf{b}_{\ref{thm_woodsquestion}}(t,k,\mathsf{eg}(H)-1)$.
    Then $G$ contains either $H$ or a $k$-grid as a minor.
\end{theorem}
\begin{proof}
    We begin by applying \zcref{thm_GMST}.
    If this yields a $H$-minor model we are done.
    Thus we may assume that we instead find a tree-decomposition $(T,\beta)$ for $G$ of adhesion at most $\mathsf{adhesion}_{\ref{thm_GMST}}(t)$ such that for every $x \in V(T)$ the torso $G_x$ of $\beta(x)$ has a strong $(\mathsf{apex}_{\ref{thm_GMST}}(t),\nicefrac{1}{2}(t-3)(t-4),\mathsf{width}_{\ref{thm_GMST}}(t))$-near embedding, with the apex set $A$, into a surface where $H$ does not embed such that $G_x - A$ is a minor of $G$.
    Note that $\mathbf{b}_{\ref{thm_woodsquestion}}(t,k) > \mathsf{adhesion}_{\ref{thm_GMST}}(t)$ and thus, according to \zcref{prop_findhighbwbag}, there exists some $t \in V(T)$ such that the torso $G_t$ of $\beta(t)$ has branchwidth at least $\mathbf{b}_{\ref{thm_woodsquestion}}(t,k)$.
    
    Let $A_t$ be the apex set belonging to $t$, let $G_t' \coloneqq G_t - A_t$, and let $\rho$ be the rendition of $G_t'$ into a surface $\Sigma$ into which $H$ does not embed that is associated with the strong near embedding of $G_t$.
    This in particular means that the genus of $\Sigma$ is below $\mathsf{eg}(H)$.
    As a consequence of \zcref{lem_branchdecompositionviacutset} and $|A_t| \leq \mathsf{apex}_{\ref{thm_GMST}}(t)$, we know that $G_t'$ still has branchwidth at least $\mathbf{b}_{\ref{thm_woodsquestion}}(t,k) - \mathsf{apex}_{\ref{thm_GMST}}(t)$.
    Let $B$ be the result of deleting $V(\sigma_{\rho}(c)) \setminus N_{\rho}(c)$ from $G_t'$ for each vortex $c \in C(\rho)$.
    We note that according to the definition of a strong near embedding $B$ is a minor of $G_t - A$, and thus in particular $G$, such that each of its components has a 2-cell-embedding in a surface into which $H$ does not embed.

    We construct $B'$ by adding a vertex $v_c$ for each vortex $c \in C(\rho)$ and making $v_c$ adjacent to all vertices in $N(c)$.
    This immediately yields a 2-cell-embedding $\phi$ of $B'$ into $\Sigma$, which can be chosen such that $v_c$ is the sole vertex in the interior of $c$ and the vertices of $N(c)$ are the only vertices found on the boundary of $c$ for all vortices $c \in C(\rho)$.
    Let $U \coloneqq V(B') \setminus V(B)$.
    Since $\rho$ has at most $\nicefrac{1}{2}(t-3)(t-4)$ vortices, we have $|U| \leq \nicefrac{1}{2}(t-3)(t-4)$.

    Suppose that $\Sigma$ is the sphere.
    Using \zcref{prop_spherewithoutvorticeshighbw} we note that $B$ still has branchwidth at least $\mathbf{b}_{\ref{thm_woodsquestion}}(t,k) - (\mathsf{apex}_{\ref{thm_GMST}}(t) + \nicefrac{1}{2}\mathsf{width}_{\ref{thm_GMST}}(t)(t-3)(t-4) )$, which is at least $4k$.
    Thus $B$ contains a $k$-grid minor according to \zcref{thm_planargrid} and $G$ also contains the same $k$-grid minor, since $B$ is a minor of $G$.

    We may therefore suppose that $\Sigma$ has positive genus.
    Suppose that $B'$ has representativity at least $4k + \nicefrac{1}{2}(t-3)(t-4)$, then $B$ must have representativity at least $4k$ and thus contain a $k$-grid according to \zcref{prop_gimmedagridfacewidth}.
    We may therefore let $\gamma$ be a non-contractible curve $\gamma$ in $\Sigma$ that intersects $\phi$ only in vertices, with $V(\gamma)$ being the set of these vertices, and $|V(\gamma)| < 4k + \nicefrac{1}{2}(t-3)(t-4)$.

    We can assume that the restriction of $\gamma$ to any vortex $c \in V(\rho)$ has one non-trivial component $\gamma'$ which is a curve in $c$ that contains $v_c \in U$ and ends on two distinct points $u^c_\gamma, w^c_\gamma$ in $\bd(c)$.
    In particular this means that for each vortex $c \in V(\rho)$ such that $v_c \in U$, if we let $(G_c,\Omega_c)$ be the vortex society of $c$ in $\rho$, the points $u^c_\gamma, w^c_\gamma$ naturally induce a partition of $\Omega_c$ into two segments $I_c^1,I_c^2$.
    Since $\rho$ has depth at most $\mathsf{width}_{\ref{thm_GMST}}$, the vortex $c$ therefore also has depth at most $\mathsf{width}_{\ref{thm_GMST}}$, which means that there exists an $I_c^1$-$I_c^2$-separator $S_c$ of order $2\mathsf{width}_{\ref{thm_GMST}}$ in $\sigma_{\rho}(c)$.
    If $c \in V(C)$ is not a vortex or $v_c \not\in U$, we set $S_c \coloneqq \emptyset$.
    Furthermore, we set $S \coloneqq (V(\gamma) \setminus U) \cup \bigcup_{c \in C(\rho)} S_c$.
    Since the breadth of $\rho$ is at most $\nicefrac{1}{2}(t-3)(t-4)$, we conclude that $|S| \leq 4k + (t-3)(t-4)\mathsf{width}_{\ref{thm_GMST}}$.
    
    Clearly, since $\gamma$ is non-contractible, the closure of the components of $\Sigma - \gamma$ consists of two surfaces $\Sigma_1, \Sigma_2$, each having $\gamma$ in their boundary.
    We now construct renditions for these two surfaces based on $\rho$ that allow us to determine which of the parts of $G_t' - S$ we should continue with. 

    For each $c \in V(\rho)$ such that $v_c \in V(\gamma)$, we add $c$ to the surfaces $\Sigma_1$ and $\Sigma_2$.
    We call the resulting surfaces $\Sigma_1'$ and $\Sigma_2'$.
    Now we take the restriction $\phi_i$ of $\phi$ to $\Sigma_i$ for each $i \in [2]$, chosen such that the vertices of $I_c^i$ are drawn by $\phi_i$ for each vortex $c \in C(\rho)$ with $v_c \in U$, and derive a vortex-free rendition $\rho_i$ from this in the natural way by enclosing each edge in a disk.
    For both $i \in [2]$, we add to $\rho_i$ all vortices $c \in C(\rho)$ with $c \subseteq \Sigma_i$.
    The graph $G_i$ that is to be embedded by $\rho_i$ is then defined as the union of the component of $B - S$ that is embedded by $\phi_i$ on $\Sigma_i$ and, for each vortex $c \in C(\rho_i)$, all components of $\sigma_{\rho}(c) - S_c$ that contain a vertex of $I_c^i$.

    Observe that $\rho_i$ has breadth at most $\nicefrac{1}{2}(t-3)(t-4)$ and depth at most $\mathsf{width}_{\ref{thm_GMST}}$ for both $i \in [2]$.
    Furthermore, $\Sigma_1$ and $\Sigma_2$ both have lesser genus than that of $\Sigma$.
    Thanks to \zcref{lem_branchdecompositionviacutset}, we also know that one of the two graphs $G_1$ and $G_2$ still has branchwidth at least 
    \[ \mathbf{b}_{\ref{thm_woodsquestion}}(t,k) - (\mathsf{apex}_{\ref{thm_GMST}}(t) + \nicefrac{1}{2}\mathsf{width}_{\ref{thm_GMST}}(t)(t-3)(t-4) + 4k + (t-3)(t-4)\mathsf{width}_{\ref{thm_GMST}}(t)) . \]
    We may now repeat this procedure by picking the graph among $G_1,G_2$ with higher branchwidth and again finishing up our proof if the surface they are embedded on is the sphere (with a few holes) or alternatively searching for non-contractible curves and reducing further.
    Since each of these steps reduces the genus of the involved surfaces further, we must finish this process after at most $\mathsf{eg}(H) - 1$ iterations, leaving us with a graph of branchwidth at least
    \begin{align*}
        \mathbf{b}_{\ref{thm_woodsquestion}}(t,k) - (\mathsf{apex}_{\ref{thm_GMST}}(t) + \nicefrac{1}{2}\mathsf{width}_{\ref{thm_GMST}}(t)(t-3)(t-4) + (\mathsf{eg}(H) - 1) (4k + (t-3)(t-4)\mathsf{width}_{\ref{thm_GMST}}(t)))
    \end{align*}
    that is embedded in the sphere (with some holes).
    Since this value is at least $4k$, this graph thus contains a $k$-grid according to \zcref{prop_gimmedagridbw} and thus we are done.
\end{proof}

Furthermore, by analysing our proof, we notice that we always find our grid in a part of the graph that is embedded on some surface.
This allows us, in the absence of a $K_t$-minor, to find the $k$-grid not only as a minor, but as an induced minor, if we double the value of $k$ for which we are searching and subsequently ``sacrifice'' half of the $2k$-grid minor we find to ensure that the result is an induced $k$-grid.
Thus, as a consequence of our efforts, we also get a proof of \zcref{thm_inducedgrid}.

If we want to turn this proof into an algorithm, we only need to be able to find the $k$-grid algorithmically, since \zcref{thm_GMST} already provides us with an $H$-minor model if it finds one.
We may again employ \zcref{lem_make_con_genus} to find our grid, once more at the cost of a somewhat worse function, which we give as follows:
\begin{align*}
    \mathbf{b}_{\ref{thm_woodsquestion_algo}}(t,k,g)   \coloneqq~& k + \max(g,0) (\nicefrac{3}{2}c_{\ref{lem_make_con_genus}}kg^{\nicefrac{5}{2}} + (t-3)(t-4)\mathsf{width}_{\ref{thm_GMST}}(t)) \ + \\
     & \max(\mathsf{adhesion}_{\ref{thm_GMST}}(t), \mathsf{apex}_{\ref{thm_GMST}}(t)) + \nicefrac{1}{2}\mathsf{width}_{\ref{thm_GMST}}(t)(t-3)(t-4) .
\end{align*}

Depending on whether we use \zcref{thm_GMST} or \zcref{thm_GMST_randomised}, our proof of \zcref{thm_woodsquestion} can thus be slightly adjusted to yield.

\begin{theorem}\label{thm_woodsquestion_algo}
    Let $k$ and $t$ be positive integers, let $H$ be a graph with $|V(H)| = t$, and let $G$ be a graph with branchwidth at least $\mathbf{b}_{\ref{thm_woodsquestion_algo}}(t,k,\mathsf{eg}(H)-1)$.
    There exists an algorithm running in time $2^{\poly(t)}kn^3m \log n$, that yields a minor model for $H$ or for a $k$-grid in $G$.
\end{theorem}

We forgo stating the explicit function for the case in which we apply \zcref{thm_GMST_randomised}.

\begin{theorem}\label{thm_woodsquestion_algo_rando}
    There exists some function $\mathbf{b}_{\ref{thm_woodsquestion_algo_rando}} \colon \mathbb{N}^3 \to \mathbb{N}$ with $\mathbf{b}_{\ref{thm_woodsquestion_algo}}(t,k,g) \in \mathcal{O}( g^{\nicefrac{7}{2}}k + gt^{11502})$ and an algorithm that, for any positive integer $k$, any graph $H$ with $t \coloneqq |V(H)|$, and any graph $G$ with $\bw(G) \geq \mathbf{b}_{\ref{thm_woodsquestion_algo_rando}}(t,k,\mathsf{eg}(H)-1)$, finds a minor model for $H$ or for a $k$-grid in $G$ with high probability in time $(t+n)^{\Ocal(1)}$.
\end{theorem}

\section{Approximating the branchwidth of \texorpdfstring{$H$}{H}-minor-free graphs}\label{sec_Hminorfree}
We now use the insights gained in the previous section to show that we can approximate both the branchwidth of $H$-minor-free graphs and the maximum size of a grid in $H$-minor-free graphs.

For the proof of the main result of this section, we again establish a function, by setting
\begin{align*}
    \mathbf{b}_{\ref{thm_approxcore}}(t,k,g)    \coloneqq~& k + (t-3)(t-4)\mathsf{width}_{\ref{thm_GMST}}(t) + 3(t-3)(t-4) \ + \\
                                                & \max(g,0)(k + (t-3)(t-4)\mathsf{width}_{\ref{thm_GMST}}(t)) + \mathsf{apex}_{\ref{thm_GMST}}(t) .
\end{align*}
As a consequence we have $\mathbf{b}_{\ref{thm_approxcore}}(t,k,g) \in \Ocal( gk + gt^{2302} )$.
We remark that in this proof we are implicitly building \emph{surface cut decompositions} (see \cite{RueST14Dynamic}) for the graphs in each of the bags.
To combine them to a branch-decomposition of the entire graph we however need to deform these decompositions further.

\begin{theorem}\label{thm_approxcore}
    Let $k$ be an integer, let $H$ be a graph with $t \coloneqq |V(H)|$, and let $G$ be a graph that does not contain $H$ as a minor.
    There exists an algorithm that determines whether $G$ has branch-width at least $k$ or branchwidth at most $\mathbf{b}_{\ref{thm_approxcore}}(t,k,\eg(H)-1)$ in time $2^{\poly(t)}kn^3m \log n$.

    In particular, in the second outcome, the algorithm returns a branch-decomposition of $G$ with width at most $\mathbf{b}_{\ref{thm_approxcore}}(t,k,\eg(H)-1)$.
\end{theorem}
\begin{proof}
    Since $G$ does not contain an $H$-minor model, we may apply \zcref{thm_GMST} to find a tree-decomposition $(T^\star,\beta)$ for $G$ with $|V(T^\star)| \in \Ocal(|V(G)|)$ and of adhesion at most $\mathsf{adhesion}_{\ref{thm_GMST}}(t)$ such that for every $x \in V(T^\star)$ the torso $G_x$ of $\beta(x)$ has a strong $(\mathsf{apex}_{\ref{thm_GMST}}(t),\nicefrac{1}{2}(t-3)(t-4),\mathsf{width}_{\ref{thm_GMST}}(t))$-near embedding, with the apex set $A$, into a surface $\Sigma$ into which $H$ does not embed such that $G_x - A$ is a minor of $G$.
    Note that according to \zcref{thm_GMST} we can find $(T^\star,\beta)$ in $2^{\poly(t)}n^3m \log n$-time, which incidentally completely dominates the runtime of all other operations we will perform, other than an additional factor of $k$ we incur later on.
    We root this tree-decomposition at some arbitrary vertex $r \in V(T^\star)$.

    \paragraph{Decomposing the bags of $(T,\beta)$ or finding large grids.}
    Let $t \in \beta(t)$ and let $G_t$ be the torso of $\beta(t)$, with the apex set $A_t$ and $G_t' \coloneqq G_t - A_t$. 
    Further, let $\rho$ be the $\Sigma$-rendition of $G_t'$ associated with the strong near embedding of $G_t$.
    We note here that we specifically assume that for every vortex $c \in C(\rho)$ and each set $X$ of the linear-decomposition of the vortex-society of $c$, we have constructed $G_t'$ such that all possible edges between vertices of $X$ are included in $G_t'$.
    This does not harm our later efforts, as this does not change the width of our vortices, but does help us in the construction of a branch-decomposition of $G$ later.
    
    Recall that according to the definition of strong near embedding the components of the graph $B$, defined as the result of removing $V(\sigma_{\rho}(c)) \setminus N_{\rho}(c)$ for all vortices $c \in C(\rho)$ from $G_t'$, each have a 2-cell embedding in a surface of genus at most that of $\Sigma$.
    We define $B'$ as in the proof of \zcref{thm_woodsquestion}, meaning that $B'$ is the result of adding a vertex $v_c$ and all edges between $v_c$ and the vertices in $N_\rho(c)$ for each vortex $c \in C(\rho)$.
    This means that in particular there exists a 2-cell-embedding $\phi$ of $B'$ into $\Sigma$ in which $v_c$ is the only vertex in the interior of the disk $c$ that is a vortex of $\rho$ and the vertices contained in $\bd(c)$ are exactly $N_\rho(c)$.
    Let $U$ be the collection of all vertices $v_c \in V(B')$ for which $c \in C(\rho)$ is a vortex.
    Due to the properties of our strong near embedding of $G_t$ we know that $|U| \leq \nicefrac{1}{2}(t-3)(t-4)$.

    Suppose that $\Sigma$ has positive genus.
    Should $B'$ have representativity at least $k + \nicefrac{1}{2}(t-3)(t-4)$, we use \zcref{prop_tanglebranchwidthduality,prop_highrepresentativitygivestangle} to conclude that $B'$ has branchwidth at least $k + \nicefrac{1}{2}(t-3)(t-4)$.
    Thus, due to the small size of $U$, the graphs $B$ and $G$ have branchwidth at least $k$ and we are done.
    We may therefore assume that we can instead apply \zcref{prop_findnoncontrcurve} to find a non-contractible curve $\gamma$ in $\Sigma$ that intersects $\phi$ only in vertices, with $V(\gamma)$ being these vertices, such that $|V(\gamma)| < k + \nicefrac{1}{2}(t-3)(t-4)$.

    As in the proof of \zcref{thm_woodsquestion}, we may assume that the intersection of $\gamma$ and each vortex $c \in C(\rho)$ for which $v_c \in V(\gamma)$ contains only one non-trivial curve $\gamma$ that contains $v_c$ and has two distinct endpoints $u^c_\gamma,w^c_\gamma$ in $\bd(c)$.
    Note that this is algorithmically feasible since we only need to know $V(\gamma)$ and the partition of $\bd(c)$ defined by the points $u^c_\gamma,w^c_\gamma$ for the coming arguments.
    Let $(G_c,\Omega_c)$ be the vortex society of $c$ in $\rho$.
    Using $u^c_\gamma,w^c_\gamma$ we may partition $\Omega_c$ into two segments $I^1_c,I^2_c$ and, since $\rho$ and thus $(G_c,\Omega_c)$ have depth at most $\mathsf{width}_{\ref{thm_GMST}}(t)$, this allows us to find an $I^1_c$-$I^2_c$-separator $S_c$ in $\sigma_\rho(c)$ of order at most $2\mathsf{width}_{\ref{thm_GMST}}(t)$.
    For any $c \in C(\rho)$ such that $v_c \not\in V(\gamma)$ we let $S_c \coloneqq \emptyset$ and define $S = (V(\gamma) \setminus U) \cup \bigcup_{c \in C(\rho) S_c}$.
    We conclude that $|S| \leq k + (t-3)(t-4)\mathsf{width}_{\ref{thm_GMST}}(t)$.

    Since $\gamma$ is non-contractible, $\Sigma - \gamma$ consists of two surfaces whose closures $\Sigma_1,\Sigma_2$ have combined genus less than the genus of $\Sigma$.
    For both $i \in [2]$ we add all disks $c \in C(\rho)$ to $\Sigma_i$ that are vortices and whose interior intersects $\Sigma_i$.
    We let the resulting surfaces be $\Sigma_1'$ and $\Sigma_2'$, and let $\phi_i$ be the restriction of $\phi$ to $\Sigma_i$ for both $i \in [2]$, chosen such that the vertices of $I_c^i$ are drawn by $\phi_i$ in $\Sigma_i$ for each vortex $c \in C(\rho)$ with $v_c \in U$.
    From this we derive a vortex-free rendition $\rho_i$ by enclosing each edge of the graph embedded by $\phi_i$ in a disk.
    For both $i \in [2]$, we then add to $\rho_i$ all vortices $c \in C(\rho)$ with $c \subseteq \Sigma_i$ and construct $G_i$, intended to be the graph decomposed by $\rho_i$, by taking the union of the component of $B - S$ that is embedded by $\phi_i$ and, for each vortex $c \in C(\rho_i)$, all components of $\sigma_{\rho}(c) - S_c$ that contain a vertex of $I_c^i$.
    For both $i \in [2]$, the $\Sigma_i$-rendition $\rho_i$ has breadth at most $\nicefrac{1}{2}(t-3)(t-4)$ and depth at most $\mathsf{width}_{\ref{thm_GMST}}(t)$ by construction.

    Let $g \leq \max(\eg(H)-1,0)$ be the genus of $\Sigma$.
    By iterating the process just outlined, we arrive either at a witness that $G$ has branchwidth at least $k$, or there exists a set $S' \subseteq V(G_t')$ with
    \[ |S'| \leq g(k + (t-3)(t-4)\mathsf{width}_{\ref{thm_GMST}}(t)) , \]
    such that we can decompose $G_t' - S'$ into graphs $G_1, \ldots , G_{g+1}$ each having a rendition $\rho_i'$ in the sphere with breadth at most $\nicefrac{1}{2}(t-3)(t-4)$ and depth at most $\mathsf{width}_{\ref{thm_GMST}}(t)$.
    We note that in case $\eg(H) = 0$, meaning that $H$ is planar, $S'$ is empty.
    This serves to explain why we use the summand $\max(g,0)(k + (t-3)(t-4)\mathsf{width}_{\ref{thm_GMST}}(t))$ in the definition of $\mathbf{b}_{\ref{thm_approxcore}}(t,k,g)$,
    
    By construction, for each $i \in [g+1]$, we can remove all sets $V(\sigma_{\rho_i'}(c)) \setminus N_{\rho_i'}(c)$, where $c \in C(\rho_i')$ is a vortex, to find a graph $B_i$ which has a 2-cell-embedding $\phi_i'$ in $\Sigma_i$ derived from $\rho_i'$ and such that $B_i$ is a minor of $G$.
    This allows us to apply \zcref{prop_catchrat} to either confirm that $B_i$ has branchwidth at least $k$, or we find a branch-decomposition of width at most $k$.
    As a result, we can find a branch-decomposition of $B_i$ with width at most $k + \mathsf{width}_{\ref{thm_GMST}}(t)(t-3)(t-4) + 3(t-3)(t-4)$ via \zcref{thm_spherewithoutvorticeshighbw}.
    
    Repeating this process for all $i \in [g+1]$ allows us to either confirm that $G$ has branchwidth at least $k$, or instead find a branch-decomposition as described above for each $B_i$.
    Applying \zcref{lem_branchdecompositionviacutset} yields a branch-decomposition $(T_t',\tau_t')$ of the entirety of $G_t$ of width at most
    \begin{align*}
        k + (t-3)(t-4)(\mathsf{width}_{\ref{thm_GMST}}(t) + 2) + |S' \cup A_t| =~& k + (t-3)(t-4)\mathsf{width}_{\ref{thm_GMST}}(t) + 3(t-3)(t-4) \ + \\
      & g(k + (t-3)(t-4)\mathsf{width}_{\ref{thm_GMST}}(t)) + \mathsf{apex}_{\ref{thm_GMST}}(t) .
    \end{align*}

    \paragraph{Piecing together a branch-decomposition for $G$.}
    We now combine the branch-decompositions of the bags of $(T^\star,\beta)$ into a branch-decomposition of $G$ with width at most $\mathbf{b}_{\ref{thm_approxcore}}(t,k,\eg(H)-1)$.
    For this purpose, let $t \in V(T^\star)$ now be arbitrary and let $t_1, \ldots , t_\ell$ be its children in $T^\star$, with $(T_t',\tau_t')$ and $(T_{t_1}',\tau_{t_1}'), \ldots, (T_{t_\ell}',\tau_{t_\ell}')$ being the associated branch-decompositions for $G_t$ and $G_{t_1}, \ldots , G_{t_\ell}$.
    We allow for $\ell = 0$ to be the case, if $t$ is a leaf of $T^\star$.
    Furthermore, we will assume that all edges represented in $(T_t',\tau_t')$ are removed from the branch-decompositions of its children.
    This clearly does not increase the width of any of these decompositions.
    
    To show that we can reconcile these decompositions, we need to take note of the particulars of a few of our proofs.
    First, note that by using the explicit constructions from the proofs of \zcref{lem_branchdecompositionviacutset} and \zcref{thm_spherewithoutvorticeshighbw} to build $(T_t',\tau_t')$, we can guarantee the following properties.
    Let $A$ be the apex set associated with $G_t$ and let $A'$ be the set we constructed above with $A \subseteq A'$ and $|A'| = \max(\eg(H) - 1,0)(k + (t-3)(t-4)\mathsf{width}_{\ref{thm_GMST}}(t))$ such that $G_t - A'$ has components $B_1^t, \ldots , B_{g'}^t$, with $g' \leq \max(\eg(H) - 1,0)$, that each have a rendition in the sphere with at most $\nicefrac{1}{2}(t-3)(t-4)$ vortices each of width at most $\mathsf{width}_{\ref{thm_GMST}}(t)$.
    For all $i \in [g']$ and for each vortex $c$ in the rendition of $B_i^t$ in the sphere, there exists a linear decomposition $(X_1^c, \ldots , X_{\ell_c}^c)$ of width at most $\mathsf{width}_{\ref{thm_GMST}}(t)$, associated with the linearisation of the vertices $v_1^c, \ldots , v_{\ell_c}^c$ on the boundary of $c$ and a partitioning $E_1^c, \ldots , E_{\ell_c}^c$ of the edges associated with $c$.
    Each vertex $v_i^c$, with $i \in [\ell_c]$ is associated with a subtree $T^c_i \subseteq T_t'$ rooted in a vertex $w_i^c$ that has two children, one being associated via $\tau_t'$ with the edge $e_i^c = v_i^cv_{i+1}^c$, or $v_i^cv_1^c$, if $i = \ell_c$, and the other being associated with a subtree of $T_t'$ whose leaves comprise those corresponding to $E_i^c$ according to $\tau_t'$ and possibly some edges associated with vertices in $A$, depending on where \zcref{lem_branchdecompositionviacutset} placed them.
    In particular, the edge incident to $w_i^c$ that lies outside of $T^c_i$ can be observed to have width at most $\mathsf{width}_{\ref{thm_GMST}}(t) + 1$, using the arguments from the proof of \zcref{thm_spherewithoutvorticeshighbw}, since the cut at worst contains all of $X_i^c$ and the one vertex in $e_i^c$ that potentially lies outside of $X_i^c$.

    Next, we make some observations on the tree-decomposition $(T^\star,\beta)$ and in particular, where we can expect $S_i \coloneqq \beta(t) \cap \beta(t_i)$ to lie with respect to $A$ and the rendition $\rho$ of $G_t' = G_t - A$, for $i \in [\ell]$.
    There are three options.
    First, we note that it is possible that $S_i \subseteq A$.
    Next, most of the proof of \zcref{thm_GMST_induction} is devoted to the case in which there exists some vortex $c \in C(\rho)$ with a linear decomposition $({X_1^c}', \ldots , {X_{\ell_c'}^c}')$ (of width at most $\mathsf{width}_{\ref{thm_GMST}}(t)$) and there exists some $j \in [\ell_c']$ such that $S_i \subseteq A \cup {X_j^c}'$.
    Recall that we have assumed above that for every bag of a linear decomposition of a vortex $c \in C(\rho)$ we have included all possible edges in $G_t'$.
    Note further that $A \subseteq A'$.
    Thus another possibility for the location of $S_i$ is that there exists a $B_j^t$, with $j \in [g']$, which has a rendition with a vortex $c$, such that there exists an $h \in [\ell_c]$ with $S_i \subseteq A' \cup X_h^c$.
    Finally, the third and last option is that there exists some cell $c \in C(\rho)$ that is not a vortex, and thus we have $|N_\rho(c)| \leq 3$, such that we have $S_i \subseteq A \cup N_\rho(c)$.
    We let $t_1, \ldots, t_\ell$ be sorted such that there exists $k,p \in [0,\ell]$ with $k < p$ and $t_1, \ldots , t_k$ fall into the first option, $t_{k+1}, \ldots , t_p$ fall into the second option, and $t_{p+1}, \ldots , t_\ell$ are associated with some non-vortex cell $c$ of $\rho$.

    We now construct our new branch-decomposition as follows.
    First, for each $i \in [\ell]$, we subdivide an arbitrary edge of $T_{t_i}'$ and let $u^s_i$ be the subdivision vertex.
    Next, we let $T^1$ be a ternary tree with $k$ leaves and the root $r^1$, in which we subdivide an arbitrary edge, with $u^1$ being the subdivision vertex, and then identify $u^s_1, \ldots , u^s_k$ via an arbitrary bijection with the leaves of $L(T)$.
    Finally, we subdivide an arbitrary edge of $T_t'$ and make the subdivision vertex adjacent to $u^1$ to form $T^1_t$.
    The bijections $\tau_t',\tau_{t_1}', \ldots , \tau_{t_k}'$ are then expanded in a straightforward fashion to $T^1_t$ to form the new branch-decomposition $(T^1_t, \tau_t^1)$.
    We observe that, since $S_i \subseteq A$, the width of this new decomposition remains firmly below our bound of $\mathbf{b}_{\ref{thm_approxcore}}(t,k,\eg(H)-1)$, as all new cuts that are introduced have width at most $|A| = \mathsf{apex}_{\ref{thm_GMST}}(t)$.

    Next, for each $i \in [k+1,p]$, let $j_i \in [g']$ be the unique index such that there exists a vortex $c_i$ in the rendition of $B^t_{j_i}$ in the sphere and there exists a set $Y^{c_i}_i$ in the linear decomposition of the vortex-society of $c_i$ with $S_i \subseteq A' \cup Y^{c_i}_i$.
    We subdivide the unique edge of $T^{c_i}_{j_i}$ incident to the leaf $l^{c_i}_{j_i}$ that is not associated with an edge contained in the vortex-society of $c_i$ and let $u^2_i$ be the subdivision vertex.
    Next, subdivide an arbitrary edge of $T_{t_i}'$ and make the subdivision vertex adjacent to $u^2_i$.
    Repeating this operation for all $i \in [k+1,p]$ yields $T^2_t$, where, if we have $Y^{c_i}_i = Y^{c_i}_{i'}$ for two distinct $i, i' \in [k+1,p]$, we continue to subdivide the edge incident to $l^{c_i}_{j_i}$.
    We again extend $\tau^1_t,\tau_{t_{k+1}}', \ldots , \tau_{t_p}'$ in a straightforward fashion to $\tau^2_t$ to get the new branch-decomposition $(T^2_t,\tau^2_t)$.
    To see that the width of $(T^2_t,\tau^2_t)$ did not increase past $\mathbf{b}_{\ref{thm_approxcore}}(t,k,\eg(H)-1)$, we first observe that the unique edge incident to $w^{c_i}_{j_i}$ that lies outside of $T^{c_i}_{j_i}$ now has width at most
    \[ \mathsf{apex}_{\ref{thm_GMST}}(t) + \max(\eg(H) - 1,0)(k + (t-3)(t-4)\mathsf{width}_{\ref{thm_GMST}}(t)) + \mathsf{width}_{\ref{thm_GMST}}(t) + 1 , \]
    as $S_i \subseteq Y^{c_i}_i \cup A'$ represents the only interaction between $G_{t_i}$ and $G_t$.
    Note in particular that this bound was already the appropriate bound for $(T_t',\tau_t')$ according to our proof of \zcref{lem_branchdecompositionviacutset}.
    Within the subtree of $T^2_t$ corresponding to $T^{c_i}_{j_i}$, the same bound holds via the same observation.
    In particular, via analogous arguments to those presented in the proof of \zcref{thm_spherewithoutvorticeshighbw}, our construction ensures that $\mathbf{b}_{\ref{thm_approxcore}}(t,k,\eg(H)-1)$ is an upper bound on the width of all cuts induced by any edge in $T^2_t$ outside of these special subtrees.

    Finally, concerning the third option, we first let $c_i$ be the non-vortex cell in $\rho$ with $S_i \subseteq A' \cup N_\rho(c_i)$ for each $i \in [p+1,\ell]$.
    We once more subdivide an arbitrary edge in $T_{t_i}'$ and let $u_i^3$ be the subdivision vertex.
    Next, let $f_i$ be an edge in $G_t'$ incident to a vertex in $N_\rho(c_i)$, preferably with both endpoints in $N_\rho(c_i)$, which is only an issue if $|N_\rho(c_i)| = 1$.
    (If no such vertex exists, we may treat $(T_{t_i}',\tau_{t_i}')$ like the first case and note that the resulting cuts have width at most $\mathsf{apex}_{\ref{thm_GMST}}(t) + \max(\eg(H) - 1,0)(k + (t-3)(t-4)\mathsf{width}_{\ref{thm_GMST}}(t))$.)
    We subdivide the unique edge incident to the leaf of $T_t^2$ that $\tau_t^2$ associates with $f_i$ and make the resulting subdivision vertex adjacent to $u_i^3$.
    Repeating this for all $i \in [p+1,\ell]$ yields our final tree $T^3_t$ and the associated bijection $\tau_t^3$ can again be derived in a straightforward fashion from $\tau^2_t,\tau_{t_{p+1}}', \ldots , \tau_{t_\ell}'$.
    In line with our previous arguments, we note that due to $|N_\rho(c_i)| \leq 3$, the most significant increase to our bounds is found in the edge incident to the former parent of the leaf of $T_t^2$ that $\tau_t^2$ associates with $f_i$.
    This edge may now have width at worst
    \[ \mathsf{apex}_{\ref{thm_GMST}}(t) + \max(\eg(H) - 1,0)(k + (t-3)(t-4)\mathsf{width}_{\ref{thm_GMST}}(t)) + 3 . \]
    As in the previous cases, the remaining edges also still have their width below our bound of $\mathbf{b}_{\ref{thm_approxcore}}(t,k,\eg(H)-1)$.
    Thus we have successfully united the branch-decompositions of $G_t$ with those of its children $G_{t_1}, \ldots , G_{t_\ell}$.
    Iterating this process bottom-up in the tree-decomposition $(T^\star,\beta)$, starting from the leaves and working its way towards the root $r$ of $T^\star$, ultimately yields a branch-decomposition of $G$ that has width at most $\mathbf{b}_{\ref{thm_approxcore}}(t,k,\eg(H)-1)$.
    This completes our proof.
\end{proof}

We note that the formulation of our statement lets us simplify our estimate on $\mathbf{b}_{\ref{thm_approxcore}}(t,k,g)$ by using \zcref{prop_ktgenus} and noting that, if $t \coloneqq |V(H)|$, we get $\mathbf{b}_{\ref{thm_approxcore}}(t,k,\eg(H)-1) \in \Ocal( \eg(H)k + t^{2304} )$.

There are a host of consequences from this theorem.
We start by noting that \zcref{thm_GMST_randomised} tells us that we can in fact perform the above computations in polynomial time via a randomised algorithm, if we accept that the answer is only returned to us with high probability and with worse bounds on the part of $\mathbf{b}_{\ref{thm_approxcore}}$ that only depends on $t$.

\begin{corollary}\label{cor_approxcore}
    Let $k$ be an integer, let $H$ be a graph with $t \coloneqq |V(H)|$, and let $G$ be a graph that does not contain $H$ as a minor.
    There exists a function $\mathbf{b}_{\ref{cor_approxcore}} \colon \mathbb{N}^3 \to \mathbb{N}$ with $\mathbf{b}_{\ref{cor_approxcore}}(t,k,g) \in \Ocal( gk + gt^{11502} )$ and an algorithm that determines whether $G$ has branchwidth at least $k$ or branchwidth at most $\mathbf{b}_{\ref{cor_approxcore}}(t,k,\eg(H)-1)$ with high probability in time $(t+n)^{\Ocal(1)}$.
\end{corollary}

Both version of this theorem of course allow us to approximate the branchwidth of any given graph $G$ that excludes a given other graph $H$ as a minor with relative ease by simply testing all $n$ possible values for it.

\begin{corollary}\label{cor_bwapprox}
    Let $H$ be a graph with $t \coloneqq |V(H)|$ and let $G$ be a graph that does not contain $H$ as a minor.
    Then
    \begin{itemize}
        \item there exists an algorithm that outputs a value $b$ with $b \leq \bw(G) \leq \mathbf{b}_{\ref{thm_approxcore}}(t,b,\eg(H)-1)$ in time $2^{\poly(t)}n^3m \log^2 n$, and

        \item there exists an algorithm that outputs a value $b$ with $b \leq \bw(G) \leq \mathbf{b}_{\ref{cor_approxcore}}(t,b,\eg(H)-1)$ in time $(t+n)^{\Ocal(1)}$ with high probability.
    \end{itemize}
\end{corollary}

Furthermore, note that we only find the $k$-grid in one specific situation within the proof of \zcref{thm_approxcore}, namely when we end up with an embedded graph on a surface with positive genus and high representativity, or with an embedded graph of high branchwidth on the sphere.
Thus, since any $K_{k^2}$ also contains a $k$-grid, instead of approximating the branchwidth of $G$, we can also approximate the size of a largest $k$-grid in $G$ using the same methods.

\begin{corollary}
    Let $H$ be a graph with $t \coloneqq |V(H)|$, let $G$ be a graph that does not contain $H$ as a minor, and let $\mathsf{k}$ be the maximum integer such that $G$ contains a $\mathsf{k}$-grid minor.
    Then
    \begin{itemize}
        \item there exists an algorithm that outputs a value $k$ with $k \leq \mathsf{k} \leq \mathbf{b}_{\ref{thm_approxcore}}(t,k,\eg(H)-1)$ in time $2^{\poly(t)}n^3m \log^2 n$, and

        \item there exists an algorithm that outputs a value $k$ with $k \leq \mathsf{k} \leq \mathbf{b}_{\ref{cor_approxcore}}(t,k,\eg(H)-1)$ in time $(t+n)^{\Ocal(1)}$ with high probability.
    \end{itemize}
\end{corollary}

In case we explicitly want to find a grid in our algorithm, we can again employ the tricks presented in \zcref{subsec:findgridbdgenus} to make our results spit out either a large grid-minor model or a branch-decomposition of small width at the cost of slightly worse functions.
We again state a pair of results, corresponding to whether we use \zcref{thm_GMST} or \zcref{thm_GMST_randomised}

\begin{theorem}\label{thm_approxcore_grid}
    There exists a function $\mathbf{b}_{\ref{thm_approxcore_grid}} \colon \mathbb{N}^3 \to \mathbb{N}$ with $\mathbf{b}_{\ref{thm_approxcore_grid}}(t,k,g) \in \mathcal{O}( kg^{\nicefrac{7}{2}} + gt^{2302})$ and an algorithm that, for any positive integer $k$, any graph $H$ with $t \coloneqq |V(H)|$, and any graph $G$ not containing $H$ as a minor, returns either a branch-decomposition of width at most $\mathbf{b}_{\ref{thm_approxcore_grid}}(t,k,\mathsf{eg}(H)-1)$ or a $k$-grid-minor model in $G$ in time $2^{\poly(t)}kn^3m \log^2 n$.
\end{theorem}
\begin{theorem}\label{thm_approxcore_grid_rando}
    There exists a function $\mathbf{b}_{\ref{thm_approxcore_grid_rando}} \colon \mathbb{N}^3 \to \mathbb{N}$ with $\mathbf{b}_{\ref{thm_approxcore_grid_rando}}(t,k,g) \in \mathcal{O}( kg^{\nicefrac{7}{2}} + gt^{11502})$ and an algorithm that, for any positive integer $k$, any graph $H$ with $t \coloneqq |V(H)|$, and any graph $G$ not containing $H$ as a minor, returns either a branch-decomposition of width at most $\mathbf{b}_{\ref{thm_approxcore_grid}}(t,k,\mathsf{eg}(H)-1)$ or a $k$-grid-minor model in $G$ with high probability in time $(t+n)^{\Ocal(1)}$.
\end{theorem}

Finally, in case we want to simply the function $\mathbf{b}_{\ref{thm_approxcore}}(t,k,g)$ and get a neater approximation at the cost of a somewhat worse runtime, we can adapt a technique presented in \cite{ThilikosW2025Approximating} to find an EPTAS of the following form.

\begin{theorem}\label{thm_eptas}
    There exists a function $f_{\ref{thm_eptas}} \colon \mathbb{N}^2 \to \mathbb{N}$ and an algorithm that, given a graph $H$ with $t \coloneqq |V(H)|$, a value $\varepsilon > 0$, and an $H$-minor-free graph $G$ as input, outputs a value $b$ where $b \leq \bw(G) \leq (\eg(H) + \varepsilon)b$ and runs in time $\Ocal(2^{\poly(t)}n^4m \log n + f_{\ref{thm_eptas}}(\varepsilon,t)n)$, where $f_{\ref{thm_eptas}}(\varepsilon,t) = 2^{(t^{2304} \cdot \nicefrac{1}{\varepsilon})}$.
\end{theorem}
\begin{proof}
    First, we run the first algorithm mentioned in \zcref{cor_bwapprox} and obtain a value $b$ such that $b \leq \bw(G) \leq \mathbf{b}_{\ref{thm_approxcore}}(t,b,\eg(H)-1)$.
    Notably, using \zcref{prop_ktgenus}, we can estimate $\mathbf{b}_{\ref{thm_approxcore}}$, such that
    \begin{align*}
        \mathbf{b}_{\ref{thm_approxcore}}(t,b,\eg(H))   =~& b + (t-3)^2(t-4)^2\mathsf{width}_{\ref{thm_GMST}}(t) \ + \\
    & \max(\eg(H)-1,0)(b + (t-3)(t-4)\mathsf{width}_{\ref{thm_GMST}}(t)) + \mathsf{apex}_{\ref{thm_GMST}}(t) + \mathsf{adhesion}_{\ref{thm_GMST}}(t) \\
    \leq~& \eg(H)b + 2t^4\mathsf{width}_{\ref{thm_GMST}}(t) + \mathsf{apex}_{\ref{thm_GMST}}(t) + \mathsf{adhesion}_{\ref{thm_GMST}}(t) .
    \end{align*}
    We let $c \coloneqq 2t^4\mathsf{width}_{\ref{thm_GMST}}(t) + \mathsf{apex}_{\ref{thm_GMST}}(t) + \mathsf{adhesion}_{\ref{thm_GMST}}(t)$.

    If $\nicefrac{c}{b} < \varepsilon$, then $c < \varepsilon b$ and we thus have $\bw(G) \leq \eg(H)b + c < (\eg(H) + \varepsilon)b$.
    This allows us to output $b$.

    Otherwise $\nicefrac{c}{b} \geq \varepsilon$, which means that $\nicefrac{c}{\varepsilon} \geq b$ and thus $\bw(G) \leq \eg(H)b + c \leq (\nicefrac{\eg(H)c}{\varepsilon} + c = c(\nicefrac{\eg(H)}{\varepsilon} + 1)$.
    Let $c' \coloneqq c(\nicefrac{\eg(H)}{\varepsilon} + 1)$.
    This allows us to use \zcref{prop_fptbw} in order to output $\bw(G)$ in $2^{c'}n$-time.

    In both cases, we return a value $b^\star$ such that $b^\star \leq \bw(G) \leq (\eg(H) + \varepsilon)b^\star$ and thus we can confirm that $f_{\ref{thm_eptas}}(\varepsilon,t) = 2^{c'}$ is chosen correctly.
\end{proof}

We can of course adapt this proof to get the following result on a randomised approximation algorithm that returns the desired value with high probability.

\begin{theorem}\label{thm_eptasrando}
    There exists a function $f_{\ref{thm_eptas}} \colon \mathbb{N}^2 \to \mathbb{N}$ and a randomised algorithm that, given a graph $H$ with $t \coloneqq |V(H)|$, a value $\varepsilon > 0$, and an $H$-minor-free graph $G$ as input, outputs with high probability a value $b$ where $b \leq \bw(G) \leq (\eg(H) + \varepsilon)b$ and runs in time $\Ocal(\poly(t+n) + f_{\ref{thm_eptas}}(\varepsilon,t)n)$, where $f_{\ref{thm_eptas}}(\varepsilon,t) = 2^{(t^{11504} \cdot \nicefrac{1}{\varepsilon})}$.
\end{theorem}

Both \zcref{thm_eptas,thm_eptasrando} can be modified to spit out branch-decompositions and grid-minor models.
Since this notably makes the approximation factor worse, we do not state or prove this here directly.
For example, in the case where we apply \zcref{thm_approxcore_grid} instead of \zcref{thm_approxcore} in the proof of \zcref{thm_eptas}, the algorithm returns a value $b$ with $b \leq \bw(G) \leq (\eg(H)^{\nicefrac{7}{2}} + \varepsilon)b$.

\bigskip

\textbf{Acknowledgements.}
The authors wish to thank anonymous referees for several helpful suggestions, which improved the presentation of the paper and some of the running times of the presented algorithms.
We further thank Tony Huynh and Bojan Mohar for helpful discussions and particularly thank Sang-il Oum for communicating David Wood's question to us.

\bibliographystyle{alphaurl}
\bibliography{literature_for_catching_rats}

\end{document}

%% file: formatting.tex
\usepackage[utf8]{inputenc}
\usepackage[T1]{fontenc}
\usepackage{lmodern}
\usepackage{alphabeta}

\usepackage{titling}
\usepackage{xspace}
\usepackage{soul} 

\usepackage[letterpaper, top=25.4mm, bottom=25.4mm, left=25.4mm, right=25.4mm, includefoot]{geometry}
\usepackage{color}
\definecolor{MidnightBlack}{rgb}{0.1,0.1,.32}
\definecolor{MidnightBlue}{rgb}{0.1,0.1,0.44}
\definecolor{Black}{rgb}{0,0, 0}
\definecolor{Blue}{rgb}{0, 0 ,1}
\definecolor{Red}{rgb}{1, 0 ,0}
\definecolor{White}{rgb}{1, 1, 1}
\definecolor{Grey}{rgb}{.6, .6, .6}
\definecolor{Mygreen}{rgb}{.0, .7, .0}
\definecolor{Yellow}{rgb}{.55,.55,0}
\definecolor{darkmagenta}{rgb}{0.30, 0.0, 0.30}
\definecolor{darkorange}{rgb}{1.0, 0.55, 0.0}
\definecolor{ao}{rgb}{1.0, 0.13, 0.32}
\definecolor{brandeisblue}{rgb}{0.0, 0.44, 1.0}

\DeclareSectionCommand[%
level=4,
indent=0pt,
beforeskip=1ex plus 1ex minus .2ex,
afterskip=-1em,
font={},
tocindent=7em,
tocnumwidth=4.1em,
counterwithin=subsubsection
]{paragraph}

\usepackage[inline]{enumitem}

\usepackage{dsfont}
\usepackage{setspace}

\usepackage{bm}
\usepackage{microtype}
\usepackage{amsmath}
\usepackage{amssymb}
\usepackage{amsfonts}
\usepackage{mathtools}
\usepackage[mathscr]{euscript}

\usepackage[hyphens]{url}
\usepackage{amsthm}

\usepackage{tikz}
\usepackage{xcolor}
\usepackage{subcaption}
\usepackage{graphicx}
\usepackage{wrapfig}

\usepackage{hyperref}
\usepackage{zref-clever}

\usepackage{nicefrac}
\usepackage[textwidth=1.5in]{todonotes}

\colorlet{myGreen}{green!50!black}
\colorlet{myLightgreen}{green}
\colorlet{myRed}{red!90!black}
\definecolor{myBlue}{rgb}{0.25, 0.0, 1.0}
\definecolor{myLightBlue}{rgb}{0.39, 0.58, 0.93}
\colorlet{myViolet}{myBlue!55!myRed}
\definecolor{myOrange}{rgb}{1.0, 0.66, 0.07}

\definecolor{CornflowerBlue}{rgb}{0.39, 0.58, 0.93}
\definecolor{DarkGoldenrod}{rgb}{0.72, 0.53, 0.04}
\definecolor{BritishRacingGreen}{rgb}{0.0, 0.26, 0.15}
\definecolor{DarkMagenta}{rgb}{0.55, 0.0, 0.55}
\definecolor{AO}{rgb}{0.0, 0.5, 0.0}
\definecolor{BostonUniversityRed}{rgb}{0.8, 0.0, 0.0}
\definecolor{myRed}{rgb}{0.8, 0.0, 0.0}
\definecolor{DarkMidnightBlue}{rgb}{0.0, 0.2, 0.4}
\definecolor{DarkTangerine}{rgb}{1.0, 0.66, 0.07}
\definecolor{AppleGreen}{rgb}{0.55, 0.71, 0.0}
\definecolor{BrightUbe}{rgb}{0.82, 0.62, 0.91}
\definecolor{Amethyst}{rgb}{0.6, 0.4, 0.8}
\definecolor{DarkGray}{rgb}{0.52, 0.52, 0.51}
\definecolor{Gray}{rgb}{0.66, 0.66, 0.66}
\definecolor{BananaYellow}{rgb}{1.0, 0.88, 0.21}
\definecolor{Amber}{rgb}{1.0, 0.75, 0.0}
\definecolor{LightGray}{rgb}{0.83, 0.83, 0.83}
\definecolor{PrincetonOrange}{rgb}{1.0, 0.56, 0.0}
\definecolor{DeepCarrotOrange}{rgb}{0.91, 0.41, 0.17}
\definecolor{CarrotOrange}{rgb}{0.93, 0.57, 0.13}
\definecolor{MidnightBlue}{rgb}{0.1, 0.1, 0.44}
\definecolor{Magenta}{rgb}{0.50, 0.0, 0.50}
\definecolor{BrightPink}{rgb}{1.0, 0.0, 0.5}
\definecolor{BrilliantRose}{rgb}{1.0, 0.33, 0.64}
\definecolor{ChromeYellow}{rgb}{1.0, 0.65, 0.0}
\definecolor{HotMagenta}{rgb}{1.0, 0.11, 0.81}
\definecolor{Amethyst}{rgb}{0.6, 0.4, 0.8}

\vfuzz \hfuzz
\setlist[itemize]{topsep=0pt,partopsep=0pt,itemsep=0pt,parsep=0pt}
\setlist[itemize,1]{label={\small\textbullet}}
\setlist[itemize,2]{label={\tiny\textbullet}}
\setlist[itemize,3]{label=$\cdot$}
\setlist[enumerate]{topsep=0pt,partopsep=0pt,itemsep=0pt,parsep=0pt}
\setlist[enumerate,1]{label=\roman*)}
\setlist[enumerate,2]{label=\alph*)}
\setlist[enumerate,3]{label=\arabic*)}

\hypersetup{
colorlinks=true,
linkcolor=AO!65!black,
citecolor=AO!65!black,
urlcolor=AppleGreen!65!black,
bookmarksopen=true,
bookmarksnumbered,
bookmarksopenlevel=2,
bookmarksdepth=3
}

\theoremstyle{definition}

\newtheorem{environment}{Environment}[section]

\newtheorem{lemma}[environment]{Lemma}
\AddToHook{env/lemma/begin}{\zcsetup{countertype={environment=lemma}}}
\zcRefTypeSetup{lemma}{
Name-sg = Lemma ,
name-sg = Lemma ,
Name-pl = Lemmas ,
name-pl = Lemmas ,
}

\newtheorem*{lemma*}{Lemma}
\AddToHook{env/lemma*/begin}{\zcsetup{countertype={environment=lemma*}}}
\zcRefTypeSetup{lemma*}{
Name-sg = Lemma ,
name-sg = Lemma ,
Name-pl = Lemmas ,
name-pl = Lemmas ,
}

\newtheorem{proposition}[environment]{Proposition}
\AddToHook{env/proposition/begin}{\zcsetup{countertype={environment=proposition}}}
\zcRefTypeSetup{proposition}{
Name-sg = Proposition ,
name-sg = Proposition ,
Name-pl = Propositions ,
name-pl = Propositions ,
}

\newtheorem{corollary}[environment]{Corollary}
\AddToHook{env/corollary/begin}{\zcsetup{countertype={environment=corollary}}}
\zcRefTypeSetup{corollary}{
Name-sg = Corollary ,
name-sg = Corollary ,
Name-pl = Corollaries ,
name-pl = Corollaries ,
}

\newtheorem{theorem}[environment]{Theorem}
\AddToHook{env/theorem/begin}{\zcsetup{countertype={environment=theorem}}}
\zcRefTypeSetup{theorem}{
Name-sg = Theorem ,
name-sg = Theorem ,
Name-pl = Theorems ,
name-pl = Theorems ,
}

\newtheorem*{theorem*}{Theorem}
\AddToHook{env/theorem*/begin}{\zcsetup{countertype={environment=theorem*}}}
\zcRefTypeSetup{theorem*}{
Name-sg = Theorem ,
name-sg = Theorem ,
Name-pl = Theorems ,
name-pl = Theorems ,
}

\AddToHook{env/conjecture/begin}{\zcsetup{countertype={environment=conjecture}}}
\zcRefTypeSetup{conjecture}{
Name-sg = Conjecture ,
name-sg = Conjecture ,
Name-pl = Conjectures ,
name-pl = Conjectures ,
}

\newtheorem*{hypothesis*}{Hypothesis}
\AddToHook{env/hypothesis*/begin}{\zcsetup{countertype={environment=hypothesis*}}}
\zcRefTypeSetup{hypothesis*}{
Name-sg = Hypothesis ,
name-sg = Hypothesis ,
Name-pl = Hypotheses ,
name-pl = Hypotheses ,
}

\AddToHook{env/observation/begin}{\zcsetup{countertype={environment=observation}}}
\zcRefTypeSetup{observation}{
Name-sg = Observation ,
name-sg = Observation ,
Name-pl = Observations ,
name-pl = Observations ,
}

\AddToHook{env/example/begin}{\zcsetup{countertype={environment=example}}}
\zcRefTypeSetup{example}{
Name-sg = Example ,
name-sg = Example ,
Name-pl = Examples ,
name-pl = Examples ,
}

\AddToHook{env/remark/begin}{\zcsetup{countertype={environment=remark}}}
\zcRefTypeSetup{remark}{
Name-sg = Remark ,
name-sg = Remark ,
Name-pl = Remarks ,
name-pl = Remarks ,
}

\zcRefTypeSetup{equation}{
Name-sg = Equation ,
name-sg = Equation ,
Name-pl = Equations ,
name-pl = Equations ,
}

\zcRefTypeSetup{chapter}{
Name-sg = Chapter ,
name-sg = Chapter ,
Name-pl = Chapters ,
name-pl = Chapters ,
}

\zcRefTypeSetup{section}{
Name-sg = Section ,
name-sg = Section ,
Name-pl = Sections ,
name-pl = Sections ,
}

\zcRefTypeSetup{algorithm}{
Name-sg = Algorithm ,
name-sg = Algorithm ,
Name-pl = Algorithms ,
name-pl = Algorithms ,
}

\AddToHook{env/notation/begin}{\zcsetup{countertype={environment=notation}}}
\zcRefTypeSetup{notation}{
Name-sg = Notation ,
name-sg = Notation ,
Name-pl = Notations ,
name-pl = Notations ,
}

\AddToHook{env/question/begin}{\zcsetup{countertype={environment=question}}}
\zcRefTypeSetup{question}{
Name-sg = Question ,
name-sg = Question ,
Name-pl = Questions ,
name-pl = Questions ,
}

\AddToHook{env/problem/begin}{\zcsetup{countertype={environment=problem}}}
\zcRefTypeSetup{problem}{
Name-sg = Problem ,
name-sg = Problem ,
Name-pl = Problems ,
name-pl = Problems ,
}

\AddToHook{env/claim/begin}{\zcsetup{countertype={environment=claim}}}
\zcRefTypeSetup{claim}{
Name-sg = Claim ,
name-sg = Claim ,
Name-pl = Claims ,
name-pl = Claims ,
}

\AddToHook{env/definition/begin}{\zcsetup{countertype={environment=definition}}}
\zcRefTypeSetup{definition}{
Name-sg = Definition ,
name-sg = Definition ,
Name-pl = Definitions ,
name-pl = Definitions ,
}

\zcRefTypeSetup{figure}{
Name-sg = Figure ,
name-sg = Figure ,
Name-pl = Figures ,
name-pl = Figures ,
}

\usetikzlibrary{calc}
\usetikzlibrary{fit}
\usetikzlibrary{decorations}
\usetikzlibrary{decorations.pathmorphing}
\usetikzlibrary{decorations.text}
\usetikzlibrary{shapes,hobby}

\tikzset{
	position/.style args={#1:#2 from #3}{
		at=($(#3)+(#1:#2)$)
	}
}

\tikzset{
  v:main/.style = {draw, circle, scale=0.8, thick,fill=black,inner sep=0.7mm},
  v:ghost/.style = {inner sep=0pt,scale=1},
  >={latex},
  e:marker/.style = {line width=8.5pt,line cap=round,opacity=0.35,color=DarkGoldenrod},
  e:main/.style = {line width=1pt},
}

\newcommand{\Ocal}{\mathcal{O}}
\newcommand{\Pcal}{\mathcal{P}}

\newcommand{\Nbbb}{\mathbb{N}}

\RequirePackage{stmaryrd}
\usepackage{textcomp}
\DeclareUnicodeCharacter{2286}{\subseteq}
\DeclareUnicodeCharacter{2192}{\ifmmode\to\else\textrightarrow\fi}
\DeclareUnicodeCharacter{2203}{\ensuremath\exists}
\DeclareUnicodeCharacter{183}{\cdot}
\DeclareUnicodeCharacter{2200}{\forall}
\DeclareUnicodeCharacter{2264}{\leq}
\DeclareUnicodeCharacter{2265}{\geq}
\DeclareUnicodeCharacter{8614}{\mathbin{\mapsto}}
\DeclareUnicodeCharacter{8656}{\Leftarrow}
\DeclareUnicodeCharacter{8657}{\Uparrow}
\DeclareUnicodeCharacter{8658}{\Rightarrow}
\DeclareUnicodeCharacter{8659}{\Downarrow}
\DeclareUnicodeCharacter{8669}{\rightsquigarrow}
\newcommand{\eqdef}{\stackrel{{\scriptsize\rm def}}{=}}
\DeclareUnicodeCharacter{8797}{\eqdef}
\DeclareUnicodeCharacter{8870}{\vdash}
\DeclareUnicodeCharacter{8873}{\Vdash}
\DeclareUnicodeCharacter{22A7}{\models}
\DeclareUnicodeCharacter{9121}{\lceil}
\DeclareUnicodeCharacter{9123}{\lfloor}
\DeclareUnicodeCharacter{9124}{\rceil}
\DeclareUnicodeCharacter{2208}{\in}
\DeclareUnicodeCharacter{9126}{\rfloor}
\DeclareUnicodeCharacter{9655}{\triangleright}
\DeclareUnicodeCharacter{9665}{\triangleleft}
\DeclareUnicodeCharacter{9671}{\diamond}
\DeclareUnicodeCharacter{9675}{\circ}
\DeclareUnicodeCharacter{10178}{\bot}
\DeclareUnicodeCharacter{10214}{} 
\DeclareUnicodeCharacter{10215}{} 
\DeclareUnicodeCharacter{10229}{\longleftarrow}
\DeclareUnicodeCharacter{10230}{\longrightarrow}
\DeclareUnicodeCharacter{10231}{\longleftrightarrow}
\DeclareUnicodeCharacter{10232}{\Longleftarrow}
\DeclareUnicodeCharacter{10233}{\Longrightarrow}
\DeclareUnicodeCharacter{10234}{\Longleftrightarrow}
\DeclareUnicodeCharacter{10236}{\longmapsto}
\DeclareUnicodeCharacter{10238}{\Longmapsto} 
\DeclareUnicodeCharacter{10503}{\Mapsto}    
\DeclareUnicodeCharacter{10971}{\mathrel{\not\hspace{-0.2em}\cap}}
\DeclareUnicodeCharacter{65294}{\ldotp}
\DeclareUnicodeCharacter{65372}{\mid}

\newcommand{\tw}{\mathsf{tw}\xspace}

\newcommand{\poly}{\textbf{poly}\xspace}

\newcommand{\polylog}{\text{polylog}\xspace}

\newcommand{\bw}{\mathsf{bw}\xspace}

\newcommand{\eg}{\mathsf{eg}\xspace}

\newcommand{\bg}{\text{$\mathsf{bg}$}\xspace}

\newcommand{\bd}{\text{$\mathsf{bd}$}\xspace}

\newcommand{\FPT}{\text{$\mathsf{FPT}$}\xspace}

\newcommand{\NP}{\text{$\mathsf{NP}$}\xspace}